\documentclass{amsart}

\usepackage{amsmath}
\usepackage{amsfonts}
\usepackage{amssymb}
\usepackage{amscd}
\usepackage{amsthm}
\usepackage{amsbsy}
\usepackage{graphicx}
\usepackage{bm}
\usepackage[mathscr]{eucal}
\def \R {\mathbb R}
\def \N{\mathbb{N}}
\newcommand{\Ricc}{\mathrm{Ric}}

\def\beq{\begin{equation}}
\def\eeq{\end{equation}}

\newcommand{\beal}[1]{\begin{eqnarray}\label{#1}}
\newcommand{\eeal}[1]{\label{#1}\end{eqnarray}}
\newcommand{\eq}[1]{\eqref{#1}}

\newcommand{\tr}{\mbox{\rm tr}}
\newcommand{\dive}{\operatorname{div}}

\newcommand{\laplaciano}[1]{\Delta_{#1}\,}     

\newcommand{\hyp}{M}
\newcommand{\threeg}{g}
\newcommand{\fourg}{\mathbf{g}}
\newcommand{\tthreeg}{\tilde\threeg}
\newcommand{\calD}{{\mathcal D}}
\newcommand{\tcalD}{\widetilde\calD}

\newcommand{\tD}{\tilde D}
\renewcommand{\div}{\mathop{\rm div}}
\newcommand{\tsigma}{{\tilde{\sigma}}}    

\newtheorem{theorem}{Theorem}[section]

\newtheorem{Theorem}[theorem]{Theorem}
\newtheorem{Corollary} [theorem]{Corollary}

\newtheorem{Proposition} [theorem] {Proposition}

\theoremstyle{definition}

\theoremstyle{remark}

\newcommand{\Hess}{\mathop{\rm Hess}}

\newcommand{\Div}{\mathrm{div}}

\begin{document}
\title{Scalar Curvature and the Einstein Constraint Equations}
\author{Justin Corvino} 
\address{Department of Mathematics, Lafayette College, Easton, PA 18042}
\email{corvinoj@lafayette.edu}
\author{Daniel Pollack}
\address{Department of Mathematics, University of Washington, Seattle, WA 98195}
\email{pollack@math.washington.edu}


\begin{abstract}  We survey some results on scalar curvature and properties of solutions to the Einstein constraint equations.  Topics include an extended discussion of asymptotically flat solutions to the constraint equations, including recent results on the geometry of the center of mass of such solutions.  We  also review methods to construct solutions to the constraint equations, including the conformal method, as well as gluing techniques for which it is important to understand both conformal and non-conformal deformations of the scalar curvature.  \end{abstract}

\maketitle

\section{Introduction}
The mathematical influence of Richard M. Schoen can be measured in many ways.  His research has fundamentally shaped geometric analysis, and his results form many cornerstones within geometry, partial differential equations and general relativity.  Evidence of his influence includes the large number of his students who continue to work in these areas.  As two of these students, the authors of this contribution are exceedingly grateful for Rick's mathematical insight and generosity.  Many of the results discussed in this article owe a great deal, directly and indirectly, to his continued mathematical vision and reach.

By the second half of the twentieth century, the theory of nonlinear partial differential equations (PDE) had developed sufficiently to allow more sophisticated analysis to be brought to bear on geometric problems.  One such early triumph was achieved by Choquet-Bruhat in 1952 \cite{ChBActa}, when she 
exploited the underlying hyperbolic nature of Einstein's field equation to establish the local well-posedness of the Cauchy problem in general relativity.
This breakthrough ushered in an era of intense study of both the (essentially hyperbolic) evolution problem for Einstein's equation, and of the set of allowable initial data for this problem, which must itself satisfy a nonlinear {\it elliptic} PDE system, the \emph{Einstein constraint equations}.  This era continues to the present day, with many remarkable recent advances and continued interest in the field from both physicists and mathematicians. 

Another watershed for the mathematical development of relativity is the celebrated work of Rick Schoen and S.-T. Yau from the late seventies on the \emph{Positive Mass Theorem}.  Not only did their work employ serious tools of geometric analysis, including partial differential equations and geometric measure theory, to resolve a question motivated by gravitational physics, but they also established a link between the positivity of the mass of an isolated gravitational system and the relationship between positive scalar curvature and topology, a topic of interest to a broad range of mathematicians.  In the early eighties, Schoen brought the Positive Mass Theorem to bear on the resolution of the famous Yamabe problem, providing more evidence to support the development of the mathematical theory of the constraint equations, and inspiring many others to do so. 

In this survey, we describe some of the aspects of the study of scalar curvature and the constraint equations that have inspired us.  Asymptotically flat solutions of the constraints are discussed at some length, including some classic proofs of fundamental results, and some relatively recent work on the geometry of such initial data sets.  We then present the conformal method for constructing solutions to the Einstein constraint equations based on the specification of certain components (\emph{free data}) of the initial data set, and indicate connections to the Yamabe problem.  In the final section, we present some results on \emph{gluing} solutions of the constraint equations together, in such a way that the resulting manifold has initial data that resembles the initial data on the original pieces, and we emphasize how both conformal and non-conformal techniques have proven fruitful, both separately and together, in constructing interesting initial data sets. 
 
Though we include some sketches of proofs and a large collection of references, we do not attempt a complete treatment of any of these topics, nor do we endeavor to make a complete bibliography.  We do hope, however, that we have included enough details so that readers will get a sense of some aspects of the field that have inspired us.  
 
 \subsubsection{Acknowledgments} 
 We would like to thank Adrian Butscher and Lan-Hsuan Huang for their comments and suggestions on an earlier draft.  We also thank Hubert Bray for his support and encouragement in the preparation of this manuscript. 
 

\subsection{Remarks on notation} 
We employ the Einstein summation convention to sum over repeated upper and lower indices.  A semi-colon preceding an index will denote a covariant derivative, whereas a comma will denote a partial derivative.  Let $dv_g$ and $d\mu_g$ be the volume and hypersurface measure, respectively, induced by a metric $g$, and let $\nu$ or $\nu_g$ be the unit normal to an oriented hypersurface.  In case $g=g_E$ is the Euclidean metric, we let the respective measures be $dv_e$ and $d\mu_e$, and let the normal be $\nu_e$.  We use the convention that the Laplacian is the trace of the Hessian. We often use subscripts to denote metric dependence on operators or norms, whereas such subscripts may be omitted for the Euclidean metric, so for example, $\Delta$ is the Euclidean Laplacian.  Finally, we note that $R(g)$ will be the scalar curvature of a metric $g$, whereas we will use ``$R$" on occasion to denote a large radius. 

\section{The Constraint Equations}
\label{CE} 
The field equation of general relativity is the Einstein equation  (with cosmological {constant $\Lambda$}),
\beq\label{eeqn} \Ricc(\fourg) -\frac12 R(\fourg)\fourg + \Lambda \fourg = 8\pi
\mathcal T\,, 
\eeq 
where $\mathcal T$ is the energy-momentum tensor, and the field $\fourg$ is a Lorentzian metric.
We say that a space-time $(\mathcal M, \fourg)$ obeys the \emph{vacuum}
Einstein equation if it obeys the Einstein equation with $\mathcal T
= 0$. 
Given a manifold $\mathcal M$, the left side of Equation~\eq{eeqn}, when expressed in terms of coordinates, forms a system of second order partial differential operators acting on the metric components $\fourg_{\mu\nu}$ which is {\em quasi-linear}, i.e. linear in the highest (second)  derivatives of the metric.  Indeed the system is linear in the second derivatives of $\fourg_{\mu\nu}$ and quadratic in the first derivatives of $\fourg_{\mu\nu}$, with coefficients which are rational functions of the components $\fourg_{\mu\nu}$.  Hence the vacuum Einstein equation constitutes a second order system of quasi-linear partial differential equations for the metric $\fourg$.
In this discussion the manifold $\mathcal M $ has been given, whereas from the evolutionary point of view, which we adopt here, solutions of the Cauchy problem for Einstein's equation yield space-times $\mathcal M$ diffeomorphic to $\R\times\hyp$, where $\hyp$ is an $n$--dimensional manifold {carrying initial data}, i.e., the initial data determines the the space-time topology and differential structure.  For a recent survey of Mathematical Relativity, which places the questions of initial data into context within the study of the Einstein equation and the Cauchy problem, we refer the interested reader to \cite{CGP}.

An \emph{initial data
set} for a vacuum space-time consists of  an $n$-dimensional
manifold $\hyp$ together with a Riemannian metric $\threeg$ and
a symmetric tensor $K$.
In the non-vacuum case we also have a
collection of non-gravitational fields which we collectively
label ${\mathcal F}$ (usually these are sections of a bundle
over $\hyp$). The vacuum
constraint equations express the vanishing of the normal (to $\hyp$)
components of the Einstein equation, and can be expressed in terms of the initial data using the Gauss-Codazzi equations. If
$\threeg$ is the metric induced on a space-like hypersurface  $\hyp$ in
a (time-oriented) Lorentzian manifold $(\mathcal M,\fourg)$, we let  $R
^i{}_{jk\ell}$ be the curvature tensor of $\threeg$.  If we let $K_{ij}$ be the second fundamental
form of $\hyp$ in $\mathcal M$, and let $ {\mathscr R} ^i{}_{jk\ell}$
be the space-time curvature tensor, the Gauss-Codazzi equations
provide the following relationships:
\beal{firsemb} &  R ^i{}_{jk\ell} = {\mathscr R}  ^i{}_{jk\ell}
+ K^i{}_\ell K_{jk} -K^i{}_k K_{j\ell}
 \;,
 &
 \\
 & D_i K_{jk} -  D_j K_{ik}=   {\mathscr R} _{ i j k\mu}n^\mu
 \;.
 &
\eeal{secemb}
Here $n$ is the unit time-like normal to the hypersurface, and the Latin indices refer to component directions tangent to the hypersurface $\hyp$. 

Contractions of \eq{firsemb}-\eq{secemb}, along with the Einstein equation, allow one to express the Einstein constraint equations  in the following form, where we have allowed for the presence of non-gravitational fields: 
\begin{eqnarray}
\Div_g K - d (\tr_g K) & = & 8\pi J \;,\label{eq:c1}\\
R(\threeg) -2\Lambda - |K|^2_\threeg + (\tr_g K)^2 & = & 16\pi \rho \;,\label{eq:c2}\\
\mathcal C({\mathcal F}, \threeg) & = & 0 \;,\label{eq:c3}
\end{eqnarray} 
where $R(\threeg)$ is the scalar curvature of the metric
$\threeg$, $J=-{\mathcal T}(n,\cdot)$ is the momentum density of the non-gravitational
fields, $\rho={\mathcal T}(n, n)$ is the energy density, and $\mathcal C({\mathcal F},
\threeg)$ denotes  the set of additional constraints that might
come from the non-gravitational part of the theory. The first
of these equations is known as the \emph{momentum constraint}
and is a vector equation on $\hyp$. The second, a scalar
equation, is referred to as the \emph{scalar}, or
\emph{Hamiltonian, constraint}, while the last are collectively
labeled the \emph{non-gravitational constraints}.  As an example, for the Einstein-Maxwell theory in 3+1 dimensions, the non-gravitational fields consist of the
electric and magnetic vector fields $E$ and $B$. In this case
we have $\rho=\frac{1}{2}(|E|^2_\threeg +|B|^2_\threeg$),
$J=(E\times B)_\threeg$, and we have the non-gravitational constraints (for vanishing charge density) $\mathcal C (E, B, \threeg)=(\dive_{\threeg} E, \dive_{\threeg} B)=0$.  From the initial data point of view, we shall regard (\ref{eq:c1})-(\ref{eq:c2}) as the \emph{definitions} of the
quantities $J$ and $\rho$.  Equations (\ref{eq:c1})-(\ref{eq:c3}) are what we shall
henceforth call the \emph{Einstein constraint equations}, or
simply the \emph{constraint equations}. 

In order for the Einstein equation to have physical and geometric relevance, one needs to either prescribe $\mathcal T$, or at least to impose some condition on it, such as an \emph{energy condition}.  Of particular significance to us here is the {\it dominant energy condition}, which requires that
\beq
\label{DECsts} 
\mbox{${\mathcal T}_{\mu\nu}X^\mu Y^\nu\ge0$ for all
 future-directed causal vector fields $X$ and $Y$.}
\eeq
At the level of the initial data, the dominant energy condition
becomes
\beq
\label{DEC}
 \rho \ge |J|_\threeg
 \;,
\eeq
where $\rho$ and $J$ are defined in \eq{eq:c1}-\eq{eq:c2}. One checks that the condition  \eq{DECsts} holds on $(\mathcal M,\fourg)$ if and only if
\eq{DEC} holds relative to each spacelike hypersurface in $\mathcal M$.  Note that for {\it maximal} ($\tr_g K\equiv 0$) initial data, with a non-negative cosmological constant, this implies that the scalar curvature is non-negative, $R(\threeg)\geq0$. This in part accounts for why we will encounter this geometric condition so often in what follows.

Equations (\ref{eq:c1})-(\ref{eq:c2}) form an underdetermined
system of partial differential equations. In the classical
vacuum setting of $n=3$ dimensions, these are locally four
equations for the twelve unknowns given by the components of the
symmetric tensors $\threeg$ and $K$.   A particular case of interest is the \emph{time-symmetric} ($K\equiv 0$) case, for which the vacuum constraint equations become simply the requirement that metric $\threeg$ has constant scalar curvature, $R(\threeg)=2\Lambda$.  In this paper we will focus primarily on the vacuum case with a zero cosmological constant. 


\section{A Tour of Asymptotically Flat Solutions}

Although many of the results below extend to higher dimensions, for clarity we will focus in this section on the \emph{three-dimensional} case ($n=3$) of the constraint equations.

One often models isolated gravitational systems by space-times satisfying the Einstein equation which admit exterior regions where the metric approaches the Minkowski metric at some rate.  In this section we will explore initial data used to model isolated systems.   At the level of the initial data, an isolated system may be modeled by a space which approaches the Euclidean metric near infinity.  To be more precise, suppose $M$ is a three-manifold with a compact subset $C\subset M$ for which $M\setminus C=\bigcup\limits_{m =1}^k E_{m}$, where the $E_{m}$ are pairwise disjoint, and each diffeomorphic to $\mathbb R^3\setminus \{ |x|\leq 1\}$.  Then we say that $(g,K)$ is \emph{asymptotically flat} with decay rate $q$ provided each $E_{m}$ admits coordinates for which we have $|\partial_x^{\alpha}(g_{ij}-\delta_{ij})(x)|=O( |x|^{-|\alpha|-q})$ and $|\partial_x^{\beta}K_{ij}(x)|=O( |x|^{-|\beta|-1-q})$, for $|\alpha|\leq \ell+1$ and $|\beta|\leq \ell$,  where $\ell\in \mathbb Z_+$ will be chosen depending on the problem at hand.  For simplicity of presentation, we can take $q=1$, though $q>\frac{1}{2}$ generally gives sufficient decay for our purposes here, cf. \cite{ba:mass}; see also the remark on the decay rate following the Positive Mass Theorem in Section \ref{sec:pmt} below. 

A family of explicit solutions of the vacuum Einstein equations which has proved important not only for physics but also for geometry is the family of \emph{Schwarzschild} space-times.  These space-times are characterized by rotational symmetry.  There are coordinates $(t,x)$ in which the Schwarzschild space-time metric takes the form 
\[\bar g_S(x)=-\left(\frac{1-\frac{m}{2|x|}}{1+\frac{m}{2|x|}}\right)^2 dt^2 +\left(1+\frac{m}{2|x|}\right)^{4}g_E\]  where $g_E$ is the Euclidean metric.  The parameter $m$ is called the \emph{mass} of the space-time.  The space-like slice $t=0$ is asymptotically flat and conformally flat with vanishing scalar curvature, and the metric $g_S(x)=\big( 1+ \frac{m}{2|x|}\big)^4 g_E$ extends to a complete metric on the set $\mathbb R^3\setminus \{ 0\}$, with two asymptotically flat ends.  We will refer to this Riemannian metric as the Schwarzschild metric below. The two-sphere $|x|=\frac{m}{2}$ inside this slice is totally geodesic, and the three-manifold has a reflection symmetry across it.  This minimal sphere is called the \emph{horizon} of the time symmetric slice.  In the Schwarzschild black-hole space-time itself, this horizon is the central leaf of the three-dimensional null hypersurface comprising the actual \emph{event horizon}.  

Another important family of solutions of the vacuum Einstein equations is given by the \emph{Kerr} space-times, which are expressed in Kerr-Schild form as
$$
 g_{\mu\nu} = \eta_{\mu\nu} +\frac{2m{\tilde r}^3}{{\tilde r}^4+a^2 z^2} \theta_\mu \theta_\nu =\eta_{\mu \nu} + O\Big(\frac{m}{r}\Big)
 \;,
$$
where $\eta=-dt^2+dx^2+dy^2+dz^2$ is the Minkowski metric, and 
$$ \theta_\mu dx^\mu = dx^0 - \frac 1 {{\tilde r}^2+a^2}\left[ {\tilde r} (xdx+ydy) + a (xdy-ydx)\right]-\frac z {\tilde r} dz\, ,$$
 where ${\tilde r}$ is defined implicitly as the solution of
the equation
$$
 {\tilde r}^4 - {\tilde r}^2 ( x^2+y^2+z^2-a^2)-a^2 z^2 =0
 \;.
$$
This space-time models a rotating isolated system, with intrinsic angular momentum $ma$ along the $z$-axis. 

In the past few decades, there has been a tremendous amount discovered about asymptotically flat solutions to the constraint equations.  One of the most striking connections was used by Schoen in his resolution of the Yamabe problem.  Indeed, if $(M,g)$ is a closed three-manifold with positive scalar curvature $R(g)>0$, then for any $p\in M$ there is a positive Green's function $G_p$ for the conformal Laplacian: $(-\Delta_g+ \frac{1}{8} R(g)) G =\delta_p$.  Let $G=4\pi G_p$.  In an appropriate coordinate system $y$ about $p$ (conformal normal coordinates, cf. \cite{lp:yam}), we have an expansion $G(y)= |y|^{-1}+ A + O(|y|)$, where $A$ is a constant.  Now consider the blowup $(\hat M =M\setminus \{p\}, \hat g=G^4 g)$.  By the formula for the scalar curvature under a conformal change of metric, we see that $R(\hat g)=0$.  A direct calculation shows that in these coordinates, $\hat g_{ij}(y)=|y|^{-4}\left[ (1+ 4A |y|)\delta_{ij}+ O(|y|^2) \right]$, so that changing coordinates by the Kelvin transform $x=y|y|^{-2}$, we have $$\hat g_{ij}(x)=|y|^4 \hat g_{ij}(y)=\left( 1+ \frac{4A}{|x|}\right)\delta_{ij}+O(|x|^{-2}).$$ Thus we see $\hat g$ is asymptotically flat with vanishing scalar curvature. The constant $A$ is related to the energy of the Ricci-flat space-time which has $(\hat M, \hat g)$ as a totally geodesic Cauchy hypersurface, as we review below. 

In the next few subsections, we will discuss normal forms near infinity for asymptotically flat solutions of the constraints, in both the time-symmetric and general cases.  We will also survey the relation between geometry and the mass and the center of mass of asymptotically flat solutions.  

\subsection{Weighted spaces}

Suppose $(M,g)$ is an asymptotically flat three-manifold.  It has proven effective to use weighted Sobolev and H\"{o}lder spaces $W^{k,p}_{-\tau}(M,g)$ and $C^{k,\alpha}_{-\beta}(M,g)$, respectively, to capture asymptotics of functions and tensors near infinity.  Let $\sigma\geq 1$ be a smooth function which equals $|x|$ near infinity in an asymptotically flat chart, and let $\gamma$ be a multi-index.  A weighted $L^p$-norm ($p\geq 1$) is then given by $$\|u\|_{L^p_{-\tau}}^p= \int_{M} (|u|\sigma^{\tau})^p \sigma^{-3} \; dv_g.$$  More generally, weighted Sobolev spaces $W^{k,p}_{-\tau}(M,g)$ are given by $$\|u\|_{W^{k,p}_{-\tau}}= \sum\limits_{|\gamma|\leq k} \|D^{\gamma}u\|_{L^p_{-\tau-|\gamma|}}=\sum_{|\gamma|\leq k}\left(\int_M ( |D^\gamma f|\sigma^{\tau+|\gamma|})^p\sigma^{-3}dv_g\right)^{1/p}.$$  We note that the weighting convention is not universal (cf. \cite{Cantor81}, \cite{CBChristodoulou}), and that we have chosen to follow \cite{ba:mass}, which uses the re-scaled measure $\sigma^{-3} \; dv_g$ in three dimensions (and, of course, $\sigma^{-n} \; dv_g$ in $n$ dimensions). 

We now recall the basic definitions of weighted H\"{o}lder spaces on $(M,g)$.  Let $\sigma$ be as above, let $\sigma(x,y)=\min(\sigma(x),\sigma(y))$, and let $d(x,y)$ be the geodesic distance with respect to $g$.  We define the weighted H\"{o}lder seminorm on functions $f$ by 
\[ [f]_{\alpha, -\beta}= \sup\limits_{x\neq y} \sigma(x,y)^{\alpha+\beta} \frac{|f(x)-f(y)|}{(d(x,y))^{\alpha}}.\]  One can naturally extend this to tensor fields, using parallel transport.  The corresponding weighted H\"{o}lder spaces $C^{k,\alpha}_{-\beta}(M,g)$ are given by those functions $f\in C^{k,\alpha}(M)$ so that the following norm is finite: \[\|f\|_{C^{k,\alpha}_{-\beta}(M,g)}:=\sum\limits_{|\gamma|\leq k} \sup\limits_{x\in M} \left(\sigma(x)^{\beta+|\gamma|}|D^{\gamma}f(x)|\right)+[D^k f]_{\alpha, -\beta-k}.\]  

The Sobolev embedding allows one to turn integral estimates into pointwise estimates.  Indeed in dimension three, we have that if $k-\frac{3}{p}\geq \alpha>0$, then $\|u\|_{C^{0,\alpha}_{-\delta}}\leq C \|u\|_{W^{k,p}_{-\delta}}$, and in fact $|u(x)|=o(|x|^{-\delta})$ as $|x|\rightarrow \infty$, cf. \cite{ba:mass}.  (The same holds in $n$ dimensions in case $k-\frac{n}{p}\geq \alpha>0$.) For large $k$, then, one can get supremum estimates on derivatives of $u$ as well. 

We note that asymptotically flat metrics may also be defined in terms of these spaces, by requiring coordinates in an end for which the components $(g_{ij}-\delta_{ij})$ belong to a weighted space, with respect to a Euclidean metric on the end.  

One reason it is nice to work with the weighted spaces is that the Laplace operator $\Delta_g$ is Fredholm (except for a discrete set of weights), and in an appropriate weight range, it is an isomorphism.  We have the following weighted elliptic estimates and regularity for the Laplace operator (for simplicity, in three dimensions) \cite{ba:mass, sw:qc}. 

\begin{Proposition} Let $\tau, \eta>0$, $\alpha\in (0,1)$, $q>n$, $1<p\leq q$, $\delta\in \mathbb R \setminus \mathbb Z$, and $k\geq 2$.  Suppose $(M,g)$ is asymptotically flat in the sense that in appropriate coordinates in any end, $g$ is uniformly equivalent to $g_E$, and $(g_{ij}-\delta_{ij})\in W^{k-1,q}_{-\tau}$.  Then $\Delta_g:W^{k,p}_{-\delta}\rightarrow W^{k-2, p}_{-\delta-2}$ is Fredholm; if $\delta\in (0,1)$, this map is an isomorphism, and there is a $C>0$ so that $\|w\|_{W^{k,p}_{-\delta}}\leq C \|\Delta_g w\|_{W^{k-2,p}_{-\delta-2}}$.  

Suppose $(g_{ij}-\delta_{ij}) \in C^{k-1,\alpha}_{-\tau}$.  Then for $\delta\in (0,1)$, $\Delta_g:C^{k,\alpha}_{-\delta}\rightarrow C^{k-2, \alpha}_{-\delta-2}$ is an isomorphism, and there is a $C>0$ so that $\|w\|_{C^{k,\alpha}_{-\delta}}\leq C \|\Delta_g w\|_{C^{k-2,\alpha}_{-\delta-2}}$.  Moreover, if $w\in C^0_{-\eta}$ and $\Delta_gw \in C^{k-2,\alpha}_{-3} \cap L^1$, then $w\in C^{k,\alpha}_{-1}$, and there is a $C>0$ so that $\|w\|_{C^{k,\alpha}_{-1}}\leq C (\|\Delta_g w\|_{C^{k-2,\alpha}_{-3}}+\|\Delta_g w\|_{L^1}).$
\end{Proposition}
We remark that the weights that are avoided in the above result are precisely those that correspond to rates of growth of Euclidean harmonic functions on $\mathbb R^3\setminus \{ 0\}$ (so that for dimension $n>3$, the \emph{exceptional} weights are given by $(-\delta)\in \{ m\in \mathbb Z: m \leq 2-n, \; \text{or } m \geq 0\}$).  

\subsection{Harmonically flat asymptotics}  Consider the time-symmetric case $K=0$ of the vacuum constraints, which reduces to the vanishing of the scalar curvature.  A particularly simple form for the geometry on an asymptotic end is for $g$ to be conformally flat with vanishing scalar curvature, which Bray has named \emph{harmonically flat}.  Suppose that $g=u^4 g_E$, in suitable coordinates $g_{ij}(x)=u^4(x)\delta_{ij}$, then the scalar curvature $R(g)=-8u^{-5}\Delta u$ is non-negative if any only if $\Delta u\leq 0$, and the scalar curvature vanishes if and only if $u$ is harmonic.  Any harmonic function $u$ which tends to 1 as $|x|\rightarrow \infty$ admits, by way of composing $u-1$ with the Kelvin transform $x\mapsto x |x|^{-2}$ (inversion in the unit sphere), an expansion in spherical harmonics, which we write 
\begin{equation}
u(x)= 1 + \frac{m}{2|x|} + \frac{\beta^1 x^1 +\beta^2x^2+\beta^3x^3}{|x|^3}+ \cdots. \label{cfexp}
\end{equation} 
The number $m$ in this expansion is called the \emph{ADM mass} \cite{adm, ba:mass}, and a simple calculation shows $$m=\frac{1}{16\pi}\lim\limits_{R\rightarrow +\infty} \int\limits_{\{|x|=R\}}
\sum\limits_{i}\left( g_{ij,i}-g_{ii,j}\right) \nu_e^j d\mu_e.$$
This limit exists and defines the ADM mass $m(g)$ for the more general asymptotically flat metrics we defined above when $R(g)\in L^1(M)$, for instance in case $(g,K)$ solves the vacuum constraints.  The mass plays an important role in the geometry of positive scalar curvature, as well as in the physics of isolated gravitational systems. 

Schoen and Yau \cite{sy:elm} showed that metrics which are harmonically flat on the ends are dense in the space of asymptotically flat metrics with non-negative scalar curvature.  

\begin{Proposition} Suppose $(M,g)$ is asymptotically flat with $R(g)\geq 0$.  For any $\epsilon>0$, there is a metric $\bar g$ within $\epsilon$ of $g$ (in a weighted norm), with $R(\bar g)\geq 0$, and which is harmonically flat near infinity in each end $E$, with $|m(g)-m(\bar g)|<\epsilon$. There is also a metric $\tilde g$ with $R(\tilde g)=0$ which is harmonically flat at each end, with $m(\tilde g)\leq m(g)+\epsilon$.  \label{prop:ha}
\end{Proposition}

\begin{proof}[Sketch of proof]
Let $0\leq \psi\leq 1$ be a smooth cutoff function so that $\psi(t)=1$ for $t<1$ and $\psi(t)=0$ for $t>2$.  Fix $r_0$, and on each end choose asymptotically flat coordinates defined for $|x|>r_0$.  For $\theta>r_0$, let $\psi_{\theta}(x)=\psi(|x|\theta^{-1})$; $\psi_{\theta}$ extends smoothly from the ends to all of $M$.   Now consider the metric $g_{\theta}(x)= \psi_{\theta}(x)g(x)+(1-\psi_{\theta}(x)) g_E(x)$.  This metric is identical to the Euclidean metric for $|x|>2\theta$, but the scalar curvature may fail to be non-negative for $\theta< |x|<2\theta$.  We now use a conformal deformation to change back to non-negative scalar curvature. Indeed, we have $R(u^4 g_{\theta})=-u^{-5}(8\Delta_{g_{\theta}} u -R(g_{\theta})u)$.  We want to impose that $R(u^4 g_{\theta})=\psi_{\theta} R(g)$; $\psi_{\theta} R(g)$ is compactly supported, and $|\psi_{\theta} R(g)|\leq C\theta^{-2-q}=C\theta^{-3}$ (in case $q=1$).  Let $u=1+v$ and set $\varXi(v)=-u^{-5}(8\Delta_{g_{\theta}} u -R(g_{\theta})u)$.  Then $\varXi$ is smooth in $v$, for $v$ in $C^{2, \alpha}_{-\delta}(M)$, or $W^{2,p}_{-\delta}(M)$ ($p>3/2$, $\delta \in(\frac{1}{2}, 1)$).  If we let $D\varXi$ be the linearization about $v=0$ ($u=1$), we get $D\varXi(w)= -8\Delta_{g_{\theta}}w -4R(g_{\theta})w$. Now, $R(g_{\theta})$ is compactly supported and (for sufficiently large $\theta$) of \emph{small} norm, so that $D\varXi$ is a small perturbation of an invertible operator (the Laplacian), and thus it is an isomorphism.   Moreover, the inverse norm is bounded uniformly in $\theta$; therefore by the inverse function theorem, for $\theta$ large enough, we can solve $\varXi(v)=\psi_{\theta} R(g)\geq 0$.  The solution $v$ will be close to zero, so that $u>0$ and $\bar g = u^4 g$ will be close to $g$, as desired, and the masses will be close, cf. \cite{sy:elm, sc:var}. 

Suppose instead we wish to deform to zero scalar curvature.  For this we study the \emph{linear} operator $\big(\Delta_g -\frac{1}{8}R(g)\big)$.  We now recall the Schoen-Yau argument that this operator is invertible on weighted spaces, say $W^{2,p}_{-\delta}$ as above, for $g$ asymptotically flat with $R(g)\geq 0$, in fact more generally in case $R(g)$ has suitably small negative part $R(g)^-$.  The operator $\big(\Delta_{g} - \frac{1}{8}R(g)\big)$ is Fredholm of index zero, since it is a compact perturbation of the Laplacian.  For any $w$ in the kernel, we have by integration by parts (using the decay of $w$) and the H\"{o}lder inequality, $\|\nabla w\|^2_{L^2(dv_{g})}\leq c \|R(g)^-\|_{L^{3/2}(dv_{g})}\|w\|_{L^{2^*}(dv_{g})}^2$, where $2^*=6$ is the Sobolev conjugate exponent to 2, in dimension three.  For  $R(g)^-$ small in $L^{3/2}$, we see that $w$ must be zero by the Sobolev inequality: $\|w\|_{L^{2^*}(dv_{g})}\leq C \|\nabla w\|_{L^2(dv_{g})}$, and therefore $\Delta_g -\frac{1}{8}R(g):W^{2,p}_{-\delta}\rightarrow W^{0,p}_{-2-\delta}$ is an isomorphism.

Now we consider the second claim in the proposition, where we wish to push the scalar curvature to zero and arrange harmonic asymptotics.  We may employ the result of the preceding paragraph to arrange zero scalar curvature.   Indeed, we let $v$ be the solution of $\Delta_g v- \frac{1}{8}R(g) v = \frac{1}{8}R(g)$, which is equivalent to $\Delta_g u-\frac{1}{8}R(g) u=0$, with $u=1+v$ tending to 1 at infinity.  Note that the constant functions $u_-=0$ and $u_+=1$ are sub- and super-solutions, respectively.  By the maximum principle and Harnack inequality, cf. \cite{gt:pde}, $0<u<1$.  Since $u$ also admits an expansion in any end given by $u(x)=1+\frac{A}{|x|}+O(|x|^{-2})$, cf. \cite{ba:mass}, we see $A<0$.  Thus the metric $u^4 g$ is asymptotically flat with vanishing scalar curvature, and moreover, $m(u^4 g)=m(g)+2A \leq m(g)$.  

Thus to prove the second claim and arrange harmonic asymptotics with the given mass bound, we may assume without loss of generality that $R(g)=0$.  The argument in the first paragraph of the proof can now be applied to find $u>0$ so that $R(u^4 g_\theta)=\psi_{\theta} R(g)=0$, for large enough $\theta$.  The result follows. \end{proof}

Bray \cite{br:th} observed that the Schoen-Yau approximation could be modified to produce ends that are \emph{precisely} Schwarzschild, preserving non-negative scalar curvature.  We now recall this argument.  Let $\phi:\mathbb R \rightarrow \mathbb R$ be a smooth, non-negative function with support $\mbox{spt} (\phi)=[-1,1]$, which is constant near the origin.  Let $\varphi(x)=\phi(|x|)$ be the associated smooth, rotationally symmetric bump function supported in the unit ball, and by scaling we may assume $\int\limits_{\mathbb R^3} \varphi(x)\; dv_e = 1$.  Let the family $\varphi_{\epsilon}(x)= \epsilon^{-3} \varphi(\frac{x}{\epsilon})$ for $\epsilon\downarrow 0$ be the corresponding approximate identity; note that $\varphi_{\epsilon}$ has unit integral, and is supported on $\{|x|\leq \epsilon\}$.  We can use $\varphi_{\epsilon}$ to mollify functions by convolution: $(\varphi_{\epsilon}\ast w)(x)=\int\limits_{y\in \mathbb R^3} \varphi_{\epsilon}(y) w(x-y)\; dv_e$.

We now state and prove Bray's proposition.  The closeness of the metrics $g$ and $\tilde g$ can be measured in a norm, or, as Bray states it, as an $\epsilon$-\emph{quasi-isometry}, i.e. for all nonzero $v\in TM$, $\frac{g(v,v)}{\tilde g (v,v)} \in (e^{-\epsilon}, e^{\epsilon})$.   

\begin{Proposition}  Suppose $(E,g)$ is an asymptotically flat end with $R(g)\geq 0$.  For any $\epsilon>0$, there is a metric $\tilde g$ with $R(\tilde g)\geq 0$, which is $\epsilon$-close to $g$, which is isometric to a (Riemannian) Schwarzschild metric near infinity in $E$, and for which $|m(g)-m(\tilde g)|<\epsilon$. \label{br:sa}
\end{Proposition}

\begin{proof} By applying Proposition \ref{prop:ha}, we may modify the metric and choose coordinates on $E$ and an $r_0>0$, so that for $|x|>r_0$, the metric has the form $g_{ij}(x)= u(x)^4 \delta_{ij}$, with $\Delta u=0$, and $u(x)= 1+ \frac{m(g)}{2|x|}+ O(|x|^{-2})$.  Now for any $R>r_0$ and $\delta >0$, consider the harmonic function $v(x)= C_1 + \frac{C_2}{|x|}$, with $C_1$ and $C_2$ chosen so that $C_1+\frac{C_2}{R}= \max\limits_{|x|=R} u(x)   + \delta$, and $C_1+\frac{C_2}{2R}= \min\limits_{|x|=2R} u(x) - \delta$.  By the choice of $C_1$ and $C_2$, we have that the function $w$ defined by $$w(x)= \begin{cases} u(x) & |x|<R \\ \min (u(x), v(x))\quad  & R\leq  |x|\leq 2R \\ v(x) & |x|>2R\end{cases} $$ is continuous.  Furthermore, note that by the expansion of $u$, we have that for any $\eta>0$, if we take $R$ sufficiently large and $\delta=\frac{\alpha}{R}$ with $\alpha$ sufficiently small, then $|C_1-1|<\eta$ and $|C_2-\frac{m(g)}{2}|<\eta$.  Since $u$ and $v$ are harmonic, and the minimum of harmonic functions is (weakly) superharmonic, we have that $w$ is superharmonic.  If we convolve $w$ with a spherically symmetric mollifier $\varphi_{\delta}$ supported in $\{ y: |y|\leq \delta\}$ as above, we produce a smooth, superharmonic function $\tilde w=\varphi_{\delta} \ast w$ which satisfies, by the mean value property,  $\tilde w (x)= u(x)$ for $|x|<R-\delta$, and $\tilde w (x)= v(x)$ for $|x|> 2R+\delta$.   Thus if we let $\tilde g_{ij}=\tilde w^4 \delta_{ij}$ on $E$, then $\tilde g$ agrees with $g$ for $r_0<|x|<R-\delta$, and $\tilde g$ is precisely Schwarzschild on $|x|>2R+\delta$, with mass $m(\tilde g)= 2 C_1C_2 \approx m(g)$. \end{proof}

\subsection {Positive mass and topological obstructions to positive scalar curvature} \label{sec:pmt}
In the late seventies, Schoen and Yau had the tremendous insight relating the proof of the positivity of the mass of an asymptotically flat metric of non-negative scalar curvature to topological obstructions to positive scalar curvature (PSC) on closed manifolds \cite{sy:pmt0}.  Lohkamp \cite{loh:ham} later observed an interesting way to further link these ideas, which we describe below.  We first recall the celebrated fundamental group obstruction to positive scalar curvature from Schoen-Yau \cite{sy:inc}.

\begin{Theorem} Suppose $M$ is a closed, orientable three-manifold.  Suppose furthermore that either $\pi_1(M)$ contains a finitely-generated non-cyclic abelian subgroup, or that $\pi_1(M)$ contains a subgroup abstractly isomorphic to the fundamental group of a closed Riemann surface of positive genus.  Then $M$ admits no metric with positive scalar curvature, and any metric on $M$ having non-negative scalar curvature is \emph{flat}. \label{psc} \end{Theorem}

We now discuss some elements of the proof.   Recall that if $\Sigma$ is a closed orientable smooth minimal hypersurface with unit normal $\nu$ and second fundamental form $A=(h_{ij})$, then the second variation of area for variation $V=\varphi \nu$ is 
\begin{equation}
 -\int\limits_{\Sigma} \varphi \mathcal L \varphi \; d\mu_g=\int\limits_{\Sigma}  \left( |\nabla_{\Sigma} \varphi|^2 - (\|A\|^2+\Ricc_g(\nu, \nu))\varphi^2\right)) d\mu_g\; ,
\end{equation}
where $\mathcal L \varphi= \Delta_{\Sigma} \varphi + (\|A\|^2+\Ricc_g(\nu, \nu))\varphi$ is called the \emph{Jacobi operator} on $\Sigma$. 

One of the key components of the proof of Theorem \ref{psc} is the following beautiful observation of Schoen and Yau, using the stability inequality $-\int\limits_{\Sigma} \varphi \mathcal L (\varphi)\; d\mu_g\geq 0$.

\begin{Proposition}  Let $(M,g)$ be a closed, oriented Riemannian three-manifold with positive scalar curvature.  Then $(M,g)$ admits no stable minimal immersion $\Sigma \looparrowright M$ from a closed Riemann surface $\Sigma$ with positive genus. 
\end{Proposition} 

\begin{proof} Suppose $\Sigma$ is a closed, oriented surface, with a stable minimal immersion to $M$.  Choose a local orthonormal frame $\{ E_1, E_2, E_3 \}$ adapted to $\Sigma$, with $E_1$, $E_2$ tangential, and $E_3=\nu$, the oriented normal to $\Sigma$.  Let $h_{ij}=g( \overline{\nabla}_{E_i} E_j , E_3 )$ denote the second fundamental form ($i,j = 1,2$) of $\Sigma$.  Let $K_{ij}$ denote the curvature (in $M$) of the $E_i$-$E_j$ two-plane, and let $K_{\Sigma}$ denote the Gauss curvature of $\Sigma$, which is \emph{half} the scalar curvature of $\Sigma$; we also have $\Ricc_g(E_3,E_3)=K_{13}+K_{23}$, and the scalar curvature $R(g)=2(K_{12}+K_{13}+K_{23})$.  The stability inequality for the variation field $V = \varphi \nu$, for $\varphi\in C^1(M)$, is written (in the local frame)
\[\int\limits_{\Sigma} \big( K_{13}+K_{23} + \sum\limits_{i,j=1}^2 h_{ij}^2 \big) \varphi^2\; d\mu_g\leq \int\limits_{\Sigma} |\nabla \varphi|^2\; d\mu_g.\]  
Now one may combine minimality, $h_{11}+h_{22}=0$, along with the Gauss equation
$K_{\Sigma}=K_{12}+h_{11}h_{22}-h_{12}^2$, to yield 
$K_{\Sigma}=K_{12}-\frac{1}{2} \sum\limits_{i,j} h_{ij}^2.$ 
Putting this into the stability inequality, with $\varphi\equiv 1$, we obtain 
\[\int\limits_{\Sigma} \Big[ \frac{1}{2}R(g) - K_{\Sigma} + \frac{1}{2} \sum\limits_{i,j=1}^2 h_{ij}^2\Big] \; d\mu_g\leq 0.\] The Gauss-Bonnet Theorem $\int\limits_{\Sigma} K_{\Sigma}\; d\mu_g = 2\pi \chi (\Sigma)$ then implies 
\begin{equation} 0<\frac{1}{2}\int\limits_{\Sigma} \big(R(g) +  \sum\limits_{i,j=1}^2 h_{ij}^2\big)\; d\mu_g \leq 2\pi \chi (\Sigma).\label{gbineq} \end{equation}   Thus $\chi(\Sigma)>0$, so $\Sigma$ has genus zero.    \end{proof}

To prove Theorem \ref{psc}, then, one seeks to prove the existence of a stable immersed surface of positive genus, given the condition on the fundamental group.  For example, given an abelian subgroup of rank two in $\pi_1(M)$, one gets a continuous map of a torus $\mathbb T^2$ into $M$ which maps $\pi_1(\mathbb{T}^2)$ onto this subgroup as follows: consider a torus as a rectangle with opposite sides suitably identified, and map these opposite sides to generators of the rank-two subgroup; the boundary of this rectangle maps to a null-homotopic curve in $M$, and so the continuous map extends to the interior of the rectangle, and hence to the torus.  (A similar procedure works for higher genus $g$ by representing the surface as a suitable quotient of a $4g$-gon.)  Amongst all maps inducing the same action on the fundamental groups (a conjugation may be invoked to keep track of the base point), one finds an energy-minimizing, hence harmonic, map.  The energy is defined with respect to a surface metric, and is conformally invariant.  By varying across conformal classes of $\mathbb{T}^2$ (or the higher genus surface), one finds a map of least energy whose action on $\pi_1(M)$ is the same as the original.  The energy-minimizer can be shown to be a stable minimal immersion.  We can now apply the preceding proposition.  See \cite{sy:inc} for details. Note that by (\ref{gbineq}), we see that one may allow $R(g)=0$, for example in case $(\mathbb T^3, g)$ were the flat torus.  By the scalar curvature deformation results of Fischer-Marsden \cite{fm:def}, $g$ must be flat, else the scalar curvature may be made positive.   

Now we recall an observation due to Lohkamp.  We give a somewhat simpler proof than in \cite{loh:ham}, using Bray's Proposition \ref{br:sa}.

\begin{Proposition} Suppose the end $(E,g)$ is harmonically flat with negative mass.  Then there is a metric on $E$ which has non-negative scalar curvature which is not identically zero, which agrees with $g$ near $\partial E$ and which is flat outside a compact set. 
\end{Proposition}

\begin{proof} Without loss in generality, we may assume by applying Bray's result that the metric $g$ is Schwarzschild near infinity in $E$: in appropriate asymptotically flat coordinates, we write $g_{ij}(x)= (1+\frac{m}{2|x|})^4 \delta_{ij}$, with $m<0$.  For $R>-\frac{m}{2}$ sufficiently large so that $E$ contains a neighborhood of $|x|=R$ in coordinates, we  consider the positive continuous function $U$ on $E$ given by $$U(x)= \min\left(1+\frac{m}{2|x|}, 1+\frac{m}{2R}\right)=\begin{cases} 1+\frac{m}{2|x|}\quad |x|\leq R \\ 1+\frac{m}{2R}\quad |x|\geq R\end{cases}\, .$$  Since $U$ is the minimum of two harmonic functions, it is superharmonic.   We use a spherically symmetric mollifier $\varphi_{\epsilon}$ as above to mollify $U$:  let $\tilde U= (\varphi_{\epsilon}\ast U)$.  $\tilde U$ is actually well-defined on $E$ for $\epsilon$ small enough, and it is smooth and positive.  Indeed for $|x|<R-\epsilon$, $\tilde U(x)= u(x)$, by the mean value property of harmonic functions, and for $|x|>R+\epsilon$, $\tilde U(x)=1+\frac{m}{2R}$ is a positive constant.  The mollification preserves superharmonicity, and thus, as it is clear that $\tilde U$ is not harmonic near $|x|=R$, there is a region where $\Delta \tilde U<0$.  Using the chosen coordinates, the metric $\tilde g_{ij}=\tilde U^4 \delta_{ij}$ has non-negative scalar curvature \emph{which is not identically zero}, and it is a flat metric outside $|x|>R+\epsilon$. 
\end{proof}

A basic version of the Positive Mass Theorem follows as a corollary of this result, using the Schoen-Yau topological obstruction to positive scalar curvature.   

\begin{Theorem}[Riemannian Positive Mass Theorem] Suppose $(M, g)$ is an asymptotically flat three-manifold with non-negative scalar curvature $R(g)\geq 0$.  Then the ADM mass of any end is non-negative. 
\end{Theorem}

\begin{proof}  From Proposition \ref{prop:ha}, we may assume that $R(g)=0$ on $M$, and that $g$ is harmonically flat on the ends.  In fact, we now recall how we can further reduce to the case where there is only one asymptotically flat end, cf. \cite[p. 202-205]{br:pen}. 

If there are several ($k>1$, say) ends of $M$, choose one, say $E$, and let $u$ be a harmonic function, $\Delta_g u=0$, with $u(x)\rightarrow 1$ as $|x|\rightarrow \infty$ in $E$, and $u(x)\rightarrow 0$ as $|x|\rightarrow \infty$ in the other ends. To find such a $u$, we fix $w\in C^{\infty}(M)$ with $w=1$ near infinity in the end $E$, and $w=0$ near infinity in the other ends.  Then $\Delta_g w \in C^{\infty}_c(M)$, and so we can solve $\Delta_g v= -\Delta_g w$, for $v$ in a weighted space (so that $v$ decays to zero in each end).  Let $u=v+w$.  We could also use constant super- and sub-solutions $u_+=1$ and $u_-=0$ and use a barrier argument.  By the maximum principle, $0<u<1$ on $M$.  

Since $R(g)=0$, the metric $u^4 g$ has vanishing scalar curvature, and near infinity in any end, it can be written $u^4g=U^4 g_E$ for $U>0$, where $\Delta U=0$.  In the ends other than $E$, $U$ tends to zero at infinity, and so if we write $U$ in spherical harmonics in these ends, we have $U(x)=\frac{c}{|x|}+\cdots$.  The higher spherical harmonics are not everywhere-positive, so that we see that $c>0$.  A simple calculation (again using the Kelvin transform $x\mapsto x|x|^{-2}$) shows that $(M,u^4 g)$ can be completed to a smooth asymptotically flat manifold $(\overline{M}, \bar g)$ by adding $k-1$ points, corresponding to compactification of all but the chosen end.   Since $u<1$, we have the mass $m(\bar g)$ is less than that of $(E,g)$, as in the proof of Proposition \ref{prop:ha}. 

If the ADM mass of $(E,g)$ were negative, then the ADM mass of $(\overline{M}, \bar g)$
would also be negative.  Now we apply the preceding proposition to assert the existence of a metric on $\overline M$ with non-negative scalar curvature which is flat outside a compact set.   We can thus consider a region $W\subset \overline M$ so that $\partial W=\{ x:|x^i|=r_0, \; i=1, 2,3\}$ is a large cube in the region where the metric is flat.  The metric thus descends to a metric with non-negative (not identically zero) scalar curvature on the quotient space $\hat M$ obtained by identifying the opposite coordinate faces in pairs.  $\hat M$ can be expressed as a connected-sum of a closed three-manifold with the torus $\mathbb T^3$, and in particular there is a copy of $\mathbb Z \oplus \mathbb Z$ in $\pi_1(\hat M)$.  We clearly have a contradiction to the Schoen-Yau obstruction to positive scalar curvature. \end{proof}

The argument above yields the fact that a metric on $\mathbb R^3$ which has non-negative scalar curvature and is Euclidean outside a compact set must in fact be globally flat.  Schoen and Yau in fact prove a strong rigidity statement: if the mass of any end vanishes, then $(M,g)$ is isometric to $(\mathbb R^3, g_E)$. From this we see that the decay rate to the Euclidean metric of an asymptotically flat metric with non-negative scalar curvature is constrained by the mass: for instance, any such metric which has the expansion in asymptotically flat coordinates $g_{ij}(x)=\delta_{ij}+O(\frac{1}{|x|^{q}})$ with $q>1$ must in fact be flat.  Finally, we note that the negative mass Schwarzschild metrics, given by $g_S(x)=(1+\frac{m}{2|x|})^4 g_E$ with $m<0$, have vanishing scalar curvature and an asymptotically flat end.  However, they become singular at $|x|=-\frac{m}{2}>0$, which is at finite geodesic distance, and such metrics are incomplete.

The minimal hypersurface proof of the Positive Mass Theorem given by Schoen and Yau \cite{sy:pmt} does not use Theorem \ref{psc} directly, but the proofs share some of the same ideas.  A calculation in asymptotically flat coordinates for a metric $g_{ij}(x)=(1+\frac{2m}{|x|}) \delta_{ij}+ O(|x|^{-2})$ yields the Christoffel symbol $\Gamma^3_{ij}= \frac{mx^3}{|x|^3}\delta_{ij}+O(|x|^{-3})$.  By tracing over $i,j=1,2$, we see that if the mass were negative, the coordinate hyperplanes $|x^3|=a$ for large enough $a$ have mean curvature vector $\vec H$ with $g(\vec H, \frac{\partial}{\partial x^3})<0$.  Hence these hyperplanes can be used as barriers for finding a stable minimal hypersurface asymptotic to a plane, by solving a Plateau problem on large cylinders with axis along the $x^3$-direction.  The stability inequality and Gauss-Bonnet yield a contradiction, in a similar manner as above, cf. \cite{sc:var}. 

\subsection{Global charges and Killing initial data (KIDs)} \label{sec:gc}

Space-times which are sufficiently asymptotically flat possess conserved quantities which are sometimes called \emph{global charges}, one example of which is the ADM mass we have already met.  In the framework of Noether's Theorem, conserved quantities correspond to symmetries.  Generic space-times may not possess symmetries, but in asymptotically flat space-times, the background Minkowskian Killing vector fields provide good enough \emph{approximate symmetries} to induce global conserved quantities, as measured by a Minkowskian observer (coordinate chart) at infinity.   These charges and their corresponding fields are as follows: mass-energy, $\frac{\partial}{\partial t}$; linear momentum, $\frac{\partial}{\partial x^i}$; center of mass, $x^i\frac{\partial}{\partial t}- t\frac{\partial}{\partial x^i}$; angular momentum, $x^j \frac{\partial}{\partial x^k}-x^k \frac{\partial}{\partial x^j}$.  The center of mass and angular momentum fit together to form a two-form $J_{\mu \nu}$, whose components give the charge associated to the Minkowski Killing field $x^\mu \frac{\partial}{\partial x^\nu}-x^\nu \frac{\partial}{\partial x^\mu}$ \cite{bo:pg, cd}.  

We find it convenient to use the momentum tensor $\pi^{ij}=K^{ij}-(\tr_g K) g^{ij}$, for which the vacuum constraint equations become 
\begin{eqnarray}
R(g)+\frac{1}{2}\left( \tr_g \pi \right)^2-|\pi|_g^2&=& 0\\
\Div_g \pi&=& 0.
\end{eqnarray}
We write the constraints map as $\Phi(g,\pi)=( R(g)+\frac{1}{2}\left( \tr_g\pi \right)^2-|\pi|_g^2, \Div_g\pi)$, so that the vacuum constraints correspond to solutions $\Phi(g,\pi)=(0,0)$. 

The relationship between linearization stability of the Einstein equation, the kernel of the formal adjoint $D\Phi^*$ of the linearization $D\Phi$, and the existence of space-time symmetries forms an important backdrop for results that will follow.  We refer to the works of Fischer-Mardsen \cite{fm:defcon}, Moncrief \cite{monc:stat}, and the more recent works \cite{ba:ps, bc:kids} for more details.  We assume $(M,g,\pi)$ satisfies the vacuum constraint equations, and consider $D\Phi$ at $(g,\pi)$.  Moncrief showed that space-time symmetries (Killing fields) correspond precisely to elements of the kernel of $D\Phi^*$.  To illustrate, we note that if $\bar g$ is the Ricci-flat metric determined by the Cauchy data on $M$, and if $\frac{\partial }{\partial t}$ is a time-like Killing field for $\bar g$, so that $M$ is a level set of $t$ with unit time-like normal $n$ and adapted coordinates $x^i$, we can write $\frac{\partial}{\partial t}=Nn+X=Nn + X^i \frac{\partial}{\partial x^i}$, and so $\bar g=-N^2 dt^2 + g_{ij}(dx^i + X^i dt)(dx^j + X^j dt)$.  Thus $(N,X)$ forms the \emph{lapse} and \emph{shift} for the stationary metric $\bar g$.  It can be shown \cite{monc:stat} that $D\Phi^*(N,X)=0$.  An element $(N,X)$ in the kernel of $D\Phi^*$ at $(g,\pi)$ is called a \emph{KID} (\emph{Killing Initial Data}).  Conversely, a non-trivial KID $(N,X)$ for $(M, g, \pi)$ can be used to generate a (not necessarily time-like) Killing field in the Einstein evolution of the Cauchy data \cite{monc:stat, bc:kids}.  

At the Minkowski data $(g_E,0)$, $D\Phi^*(N,X)=\left( -(\Delta N)g_E +\Hess N, -\frac{1}{2} L_X g_E\right)$, where $L_X g_E$ is the Lie derivative.  The kernel $K$ of $D\Phi^*$ is thus the direct sum of $K_0=\mbox{span}\{1, x^1, x^2, x^3\}$ together with the space of Killing fields of $g_E$, which of course is spanned by the generators of rotations and translations.

As it turns out, asymptotic flatness alone is not enough to guarantee a well-defined angular momentum, cf. \cite{bo:pg, lan:spa}.  Regge and Teitelboim \cite{rt} proposed asymptotic conditions sufficient to guarantee a well-defined angular momentum and center of mass for solutions to the constraints, namely that in suitable asymptotically flat coordinates the following estimates also hold (with the decay rate $q=1$ for simplicity):
\begin{equation} \label{rt}
\Big|\partial_x^{\alpha}\Big(g_{ij}( x)-g_{ij}(- x)\Big)\Big|=O(| x|^{-|\alpha|-2})
 , \;
\Big |\partial_x^{\beta}\Big(K_{ij}( x)+K_{ij}(- x)\Big)\Big|=O(|x|^{-|\beta|-3}).
\end{equation}
These conditions impose approximate parity symmetry on the data $(g,K)$ in an asymptotically flat end.  The Kerr space-times admit coordinates (such as Kerr-Schild) for which the constant time slices (and corresponding boosted slices) satisfy (\ref{rt}).  The space of initial data satisfying (\ref{rt}) is known to be dense in the space of vacuum asymptotically flat data \cite{cs:ak}, as we discuss in the next section. 

Using such a coordinate system, we can compute the energy and linear and angular momenta using flux integrals at infinity.  We let $Y_{(i)}$ be the Euclidean rotational Killing fields, e.g. $Y_{(3)}=x^1 \frac{\partial}{\partial x^2}-x^2\frac{\partial}{\partial x^1}$, we let $\nu_e$ be the Euclidean outward normal, and recall that the Einstein convention is in force.  Furthermore, let the integral over ${S_{\infty}}$ denote the limit  of integrals over $\{|x|=R\}$ as $R\rightarrow +\infty$. The flux integrals for the energy and momenta are given by  
\begin{eqnarray}
m&=&\frac{1}{16\pi}\int\limits_{S_{\infty}}
\sum\limits_{i}\left( g_{ij,i}-g_{ii,j}\right) \nu_e^j d\mu_e \nonumber \\
P_i&=&\frac{1}{8\pi}\int\limits_{S_{\infty}}
(K_{ij}-K^{\ell}{}_{\ell}g_{ij})\nu_e^j d\mu_e=\frac{1}{8\pi}\int\limits_{S_{\infty}}
\pi_{ij}\nu_e^j d\mu_e \label{adm} \\
J_i&=&\frac{1}{8\pi}\int\limits_{S_{\infty}}
(K_{jk}-K^{\ell}{}_{\ell}g_{jk})Y^k_{(i)}\nu_e^j d\mu_e =\frac{1}{8\pi}\int\limits_{S_{\infty}}
\pi_{jk}Y^k_{(i)}\nu_e^j d\mu_e\nonumber\\
 mc^{k} &=&\frac{1}{16\pi}\int\limits_{S_{\infty}}
\Big[\sum\limits_{i} x^k \left( g_{ij,i}-g_{ii,j}\right) \nu_e^j  -
\sum\limits_{i}\big(g_{i\ell}\delta^{k\ell} \nu_e^i- g_{ii}\nu_e^k \big)\Big]  d\mu_e . \label{eq:cm}
 \end{eqnarray}
In the last term, one may replace $g_{ab}$ in the 
integrand by $(g_{ab}-\delta_{ab})$.    Taken together, these
give a set of ten \emph{Poincar\'{e} charges} associated to the
end.  We remark that, from the physical point of view, the quantity $m$, which corresponds to the approximate symmetry generated by $\frac{\partial}{\partial t}$, might more properly be termed the \emph{energy} $E$ instead of the mass, as the charges $(E, P^i)$ are the components of the  \emph{energy-momentum four-vector}.  From the point of the view of the space-time, this vector is Lorentz-covariant, and the square of the rest mass is $E^2-\sum\limits_{i=1}^3 (P^i)^2$.   We stick to the notation above for consistency with the Riemannian case. 

It is instructive to view these charge integrals as arising from integrating the
constraint functions against elements of the cokernel of the linearized
constraint operator.  By linearizing at the Minkowski data $(g_E,0)$, we
have $\Phi(g_E+h,\pi)= D\Phi(h, \pi)+ Q(h,\pi)$, where a simple expansion (with $q=1$) yields an estimate of the quadratic error term $Q(h,\pi)=O(|x|^{-4})$.
In Euclidean coordinates at
the Minkowski data, $D\Phi(h,\pi)=\Big( \sum\limits_{i,j} (h_{ij, ij}-h_{ii,
jj}), \sum\limits_{j}\pi_{ij,j}\Big).$  Thus for any KID $(N,X)$, i.e. for any vector and scalar pair $(N,X)$ which satisfies $D\Phi^*(N,X)=(0,0)$, we have as a consequence of
integration by parts that 
\begin{equation} \label{bdyint}
\int\limits_{ A(R_0,R) } (N,X)\cdot \Phi(g_E+h, \pi) dv_e = \mathcal B(R)-\mathcal B(R_0) +  \int\limits_{A(R_0,R) } (N,X)\cdot Q(h,\pi) dv_e
\end{equation}
where $A(R_0,R)=\{ R_0\leq |x|\leq R\}$, and 
$$\mathcal{B}(R)= \int_{\{ |x|=R \} }
\Big( \pi_{ij}X^i+   N \sum\limits_{i}(h_{ij,i}-h_{ii,j}) - \sum\limits_{i}
(N_{,i}h_{ij} - N_{,j} h_{ii})  \Big)\nu_e^j\; d\mu_e .$$ By
letting $X$ be a Euclidean Killing vector field, or letting $N$
be a constant or a coordinate function $x^k$, we can easily
relate $\mathcal B(R)$ to one of the above surface integrals
defining the ADM energy-momenta.

\subsubsection{Comparison to Newtonian theory}
We now compute the mass and center of mass in a harmonically flat end, and compare to classical Newtonian theory.  We expand the harmonic conformal factor $u(x)= 1+\frac{A}{|x|}+\frac{\beta\cdot x}{|x|^3}+O(|x|^{-3})$, with $\beta\cdot x=\beta^1 x^1 +\beta^2x^2+\beta^3x^3$, and we note the error term picks up an extra $|x|^{-1}$-decay with each successive derivative.  Using this in the flux integrals for mass and center, we obtain
\begin{align*}
\frac{1}{16\pi}\int\limits_{\{|x|=R\} } \sum\limits_{i} (g_{ij,i}-g_{ii,j})\nu_e^j\; d\mu_e =2A +O(\tfrac{1}{R})\\
\frac{1}{16\pi}\int\limits_{\{|x|=R\}}\sum\limits_{i}\Big[ 
 x^k \left( g_{ij,i}-g_{ii,j}\right) \nu_e^j  - 
\big(g_{i\ell}\delta^{k\ell} \nu_e^i- g_{ii}\nu_e^k \big) \Big] d\mu_e &=2 \beta^k+O(\tfrac{1}{R})\; .
\end{align*}
Thus we see $m=2A$ and $mc^k=2Ac^k=2\beta^k$.  

We note that if we translate asymptotically flat coordinates $y^k=x^k+a^k$, then the center of mass integral becomes $m(c^k+a^k)$ as expected.  In the harmonically flat case, we note that the expansion transforms under translation as $u(y-a)=1+\frac{A}{|y|} +\frac{\tilde\beta\cdot y}{|y|^3}+ O(|y|^{-3})$, where $\tilde\beta^k=\beta^k+Aa^k$.  So the definition of $c^k$ corresponds to the translation which makes the $|x|^{-2}$-terms in the expansion vanish: $u(y+c)=1+\frac{A}{|y|} + O(|y|^{-3})$.  

In the Newtonian setting, we can study gravitational systems in terms of the mass density $\rho$, with decay conditions on $\rho$ corresponding to the system being isolated; for example, we might take $\rho$ to be compactly supported.  This gives rise to a gravitational potential $\phi$ which satisfies $\Delta \phi= 4\pi \rho$ (we have taken Newton's constant $G=1$).   In case $\rho$ is compactly supported, $\phi$ is harmonic near infinity; we take $\phi(x)\rightarrow 0$ as $|x|\rightarrow \infty$, and thus $\phi$ can be expanded as a series in spherical harmonics, which yields $\phi(x)= -\frac{m}{|x|}- \frac{\beta\cdot x}{|x|^3}+ O(|x|^{-3})$.  Note that the total mass $\int_{\mathbb{R}^3} \rho \; dv_e $ and the moments $\int_{\mathbb{R}^3} x^k \rho\; dv_e$ defining the center of mass of the system can be written as flux integrals in terms of the potential, analogous to those above: 
\begin{align*}
\int_{\mathbb{R}^3} \rho \; dv_e = \frac{1}{4\pi} \int_{\mathbb{R}^3} \Delta \phi \; dv_e&= \frac{1}{4\pi}\lim\limits_{R\rightarrow +\infty}\int_{\{|x|=R\} } \frac{\partial \phi}{\partial r} \; d\mu_e = m\\
\int_{\mathbb{R}^3} x^k \rho\; dv_e = \frac{1}{4\pi} \int_{\mathbb{R}^3} x^k \Delta \phi \; dv_e 
&= \frac{1}{4\pi}\lim\limits_{R\rightarrow +\infty} \int_{\{ |x|=R\} }\big( x^k \frac{\partial \phi}{\partial r}- \frac{x^k}{R} \; \phi\big)d\mu_e=   \beta^k.
\end{align*}  
We see then that the center of mass $c^k$ can be defined when $m\neq 0$ by $mc^k=\beta^k$.  

The density $\rho$ generates the potential $\phi$.  Analogously we can consider perturbations of the Euclidean metric, say $\gamma=g_E+h$, where $h$ is sufficiently small and for simplicity has compact support.  We generate from this a potential function $u>0$ tending to $1$ at infinity, for which $R(u^4 \gamma)=0$, i.e. $\Delta_{\gamma} u = \tfrac{1}{8} R(\gamma) u$.  Near infinity $u$ is harmonic, so that for large $|x|$, $u(x)= 1+ \frac{m}{2|x|}+\frac{\beta\cdot x}{|x|^3}+O(|x|^{-3})$, and we have (compare to the Newtonian formulas above)
\begin{eqnarray*}
 m  &=& - \frac{1}{2\pi} \lim\limits_{R\rightarrow +\infty} \int\limits_{\{|x|=R\}} \frac{\partial u}{\partial r}\; d\mu_e = -\frac{1}{2\pi}\int\limits_{\mathbb{R}^3} \Delta_{\gamma} u \; dv_{\gamma}= - \frac{1}{16\pi}\int_{\mathbb{R}^3} R(\gamma)u \; dv_{\gamma}\\
 \beta^k&=&-\frac{1}{4\pi} \int_{\mathbb{R}^3} y^k \Delta_\gamma u \; dv_\gamma =-\frac{1}{32\pi}   \int_{\mathbb{R}^3} y^k R(\gamma) u \; dv_\gamma\; ;
\end{eqnarray*}
in the preceding integral, we have taken global harmonic coordinates $y^i$ for the metric $\gamma$ (so that $\Delta_\gamma y^i=0$, which simplifies Green's formula above), with $|x-y|=O(|x|^{-1/2})$, valid for $h$ small \cite[Theorem 3.1]{ba:mass}.

\subsection{Harmonic asymptotics}    

In this section we discuss an extension of harmonically flat asymptotics to the setting of the full constraint equations, in which we study both the metric $g$ and second fundamental form $K$ of the space-like slice.  These special \emph{harmonic asymptotic} conditions were proposed by Schoen in conjunction with a proof of the Positive Mass Theorem, cf. \cite{sch:mcsurv}, but have already proven useful in other contexts \cite{cs:ak, lan:cm}.  

Let ${\mathcal L}_gX=L_X g-\Div_g(X)g$, where $(L_X g)_{ij}=X_{i;j}+X_{j;i}$ is the Lie derivative operator. Note that $\tr_g(\mathcal L_g X)= -\Div_g (X)$, in contrast to the trace-free conformal Killing operator $\mathcal D_g(X)=L_X g - \frac{2}{3} \Div_g(X)g$ that is used in the conformal method for solving the constraints (see Section \ref{ConfMet}).  The harmonic asymptotic conditions require that outside a compact set
there exist a positive function $u$ which tends to 1 at infinity, and a vector field
$X$ which tends to 0 at infinity, so that $g= u^4 g_E$ and $\pi =u^2\mathcal L_{g_E} X$, and so that in suitable asymptotically flat coordinates $g_{ij}=u^4\delta_{ij}$ and
$\pi_{ij}=u^2 (X_{i,j}+X_{j,i}-X_{,k}^k \delta_{ij})$, where $X^i=X_i$.  For such data, the vacuum constraint equations become (with $\mathcal L = \mathcal L_{g_E}$, $\Delta=\Delta_{g_E}$)
\begin{align}
8 \Delta u- u\big( -|{\mathcal L}X|^2+\tfrac{1}{2}(\tr({\mathcal L}X))^2\big) &= 0\label{haxcon1}\\ 
 \Delta X_i+4u^{-1} u_{,j}({\mathcal L}X)^j_{i}-2u^{-1}u_{,i} \tr({\mathcal L}X)&=0. \label{haxcon2}
 \end{align}
From these equations one can immediately see the advantage of using $\mathcal L_g$: its divergence has leading order term which is precisely the Laplacian of the components of $X$.  This fact is used in the proof of the following density theorem \cite{cs:ak}.  
\begin{Theorem} Let $\delta \in (\frac{1}{2}, 1)$ and $p>\frac{3}{2}$, and suppose $(M,g, \pi)$ is a vacuum initial data set with $(g_{ij}-\delta_{ij}, \pi_{ij})\in W^{2,p}_{-\delta} \times W^{1,p}_{-1-\delta}$.  Given any $\epsilon>0$, there is a vacuum initial data set $(\bar g, \bar \pi)$ with \emph{harmonic asymptotics} and within $\epsilon$ of $(g, \pi)$ in the $W^{2,p}_{-\delta} \times W^{1,p}_{-1-\delta}$ norm, and so that the mass and linear momentum of $(\bar g, \bar \pi)$ are within $\epsilon$ of those of $(g, \pi)$. \end{Theorem}

The proof is obscured, in comparison with that of the Schoen-Yau harmonic asymptotic approximation sketched above, by the fact that the linearization $D\Phi$ at a general asymptotically flat solution of the constraints is somewhat complicated, as is that of the constraint map for harmonic asymptotics (the left-hand side of (\ref{haxcon1})-(\ref{haxcon2}) above).  The method of proof uses the surjectivity of $D\Phi$ in appropriate spaces, and has already been employed in other works, cf. \cite{ba:ps, lan:cm}.

An important feature of these asymptotic conditions is that (\ref{rt}) holds, and moreover, the total energy-momentum and angular momentum are directly encoded in the asymptotics of $(u,X)$, and thus these 
conserved quantities directly affect the asymptotic geometry.  In fact, one solves for $(u-1, X)$ in a weighted space, and a standard expansion for the solution of Poisson's equation, cf., e.g., \cite{ba:mass}, yields the following expansions near infinity:
\begin{equation} 
u(x)= 1 + \frac{A}{|x|}+O(|x|^{-2}), \qquad X^i(x)= \frac{B^i}{|x|}+O(|x|^{-2}). \label{haxp}
\end{equation}
These leading coefficients are, up to constant factors, the ADM mass/energy and the linear momentum of the end.  Indeed, it is easily shown as above that $A=m/2$.  A quick computation yields $$\pi_{ij}=-\frac{B^i}{|x|^2}\frac{x^j}{|x|}- \frac{B^j}{|x|^2}\frac{x^i}{|x|}+ \sum\limits_{k=1}^3 \frac{B^k}{|x|^2}\frac{x^k}{|x|} \delta_{ij}+O(|x|^{-3}).$$ Thus on a large coordinate sphere $|x|=R$ with outward Euclidean normal $\nu_e=\frac{x}{|x|}$, we have $\pi_{ij} \nu_e^j = \pi_{ij} \frac{x^j}{|x|}=-\frac{B^i}{|x|^2}+ O(|x|^{-3}).$  Thus we see the linear momentum vector is $P= - \frac{B^i}{2}\frac{\partial }{\partial x^i}$.  

Though in general we cannot expand $(u,X)$ further in spherical harmonics, we can in fact expand the \emph{odd} parts of $u$ and $X$ to one higher order.  From the flux integrals for angular momentum and center of mass, we see that the terms at the next order of the expansion that will contribute to the limiting values of the integrals are precisely the \emph{odd} parts.  From the constraint equations (\ref{haxcon1})-(\ref{haxcon2}) for $(u,X)$, using (\ref{haxp}), we see $\Delta u^{\text{odd}}=O(|x|^{-5})$ and $\Delta X^{\text{odd}}=O(|x|^{-5})$, so that $u^{\text{odd}}(x)=\frac{\beta\cdot x}{|x|^{3}}+O(|x|^{-3})$ and $(X^i)^{\text{odd}}(x)=\frac{d_{(i)}\cdot x}{|x|^{3}}+O(|x|^{-3})$.  We have seen how the vector $\beta$ is related to the center of mass, and similarly, the angular momentum can be written in terms of $d_{(i)}$, e.g. $J_3= \frac{1}{2}(d_{(1)}^2 - d_{(2)}^1)$.   We also remark the following: if we shift the coordinates to $y=x-c$, let $\tilde X(y)=X(y+c)$ and expand, we get $(\tilde X^i)^{\text{odd}}(y)=\frac{\tilde{d}_{(i)}\cdot y}{|y|^3}+O(|y|^{-3})$, where $\tilde{d}_{(i)}^k= d_{(i)}^k - c^kB^i =d_{(i)}^k +2c^kP^i $.  Thus we see that the angular momentum transforms as $J_i\mapsto J_i - (c\times P)_i$, as expected. 

\subsection{Geometry of the center of mass}
A natural question is whether the notion of center of mass can be described in a geometrically interesting and intrinsic way.  This may be particularly intriguing in the \emph{vacuum} case, where the only energy comes from the gravitational field (metric).  In this section, we discuss approaches to answer this question, and the relation between them. 
 
\subsubsection{Center of mass and conformal symmetries}
As we noted above, the mass and center of mass are related to asymptotic symmetries of the space.   Based on a suggestion of R. Schoen, L.-H. Huang \cite{lan:cm} employed a computation of the mass and center using approximate conformal symmetries, as we now recall.   The connection is motivated in part by the following generalized Pohozaev identity from \cite{sc:singyam}.

\begin{Proposition} Let $(M,g)$ be a compact Riemannian $n$-manifold with smooth boundary $\partial M$.  Suppose $X$ is a conformal Killing field on $M$.  Then we have the following identity, where $\nu$ is the outward unit normal to $\partial M$:  \begin{eqnarray} \int\limits_M X(R(g)) dv_g=\frac{2n}{n-2} \int\limits_{\partial M}(\Ricc(g)-\frac{R(g)}{n}g)(X,\nu)d\mu_g
\; . \label{eq:poho} \end{eqnarray} 
\end{Proposition}

To illustrate, we note than in the case of harmonically flat asymptotics, the mass and center can be associated to exact conformal Killing fields in the asymptotic region.  Since dilations and inversions are conformal isometries of the flat metric, we can represent dilation or translation of infinity by conjugating dilations or translations with the Kelvin transform to obtain the following conformal isometries near infinity, for any $a>0$, and any $y\in \mathbb R^3$ (computed in the harmonically flat coordinates $x$, for which $g_{ij}=u^4 \delta_{ij}$): \[ x\mapsto \frac{x}{|x|^2}\mapsto \frac{ax}{|x|^2}\mapsto \frac{x}{a}, \quad \quad x \mapsto \frac{x}{|x|^2}+y \mapsto \left|\frac{x}{|x|^2}+y\right|^{-2}\left(\frac{x}{|x|^2}+y\right) .\] By taking the derivative of the first map with respect to $a$ at $a=1$, we see that the infinitesimal generator of this family of isometries is the conformal Killing field $\frac{1}{2}X_{(0)}=-x^i\frac{\partial}{\partial x^i}$.  Differentiating the second map at $y=0$ yields the conformal Killing fields $X_{(\ell)}=|x|^2\frac{\partial}{\partial x^\ell}-2x^\ell x^j\frac{\partial}{\partial x^j}$, for $\ell=1,2,3$. 

We compute the integrals of $\Ricc_g(X_{(\ell)},\nu)$ over a large sphere $\{ |x|=R\}$ in the asymptotically flat region. We take the $g$-normal $\nu$ pointing toward infinity, and let $\nu_e=u^2 \nu$ be the Euclidean normal, so that 
\begin{eqnarray*}
\int\limits_{\{ |x|=R\}}\Ricc_g(X_{(0)},\nu) d\mu_g&=& -2\int\limits_{\{ |x|=R\}} \frac{u^2}{|x|}x^jx^kR_{jk}\; d\mu_e=16\pi m+O(R^{-1})\; ,\\
\int\limits_{\{ |x|=R\}}\Ricc_g(X_{(\ell)},\nu) d\mu_g &=& 32\pi \beta^\ell + O(R^{-1})= 16\pi m c^\ell+O(R^{-1})\; ,
\end{eqnarray*}
where we used $R_{jk}=(\Ricc(g))_{jk}=-u^{-2}\big( 2u u_{,jk}+6u_{,j} u_{,k} -2\delta_{jk}|\nabla u|_e^2\big)$ (see \cite{cw:cm} for details).  If $g$ is harmonically flat near infinity, then by applying (\ref{eq:poho}) to a large annular region, we get an asymptotic conservation statement for the mass integral across the boundaries of the region.  We note that in the center of mass integrals, the leading-order terms from the expansion are not absolutely integrable, but they have the correct parity so that their surface integrals vanish. 

More generally, the ADM mass can be expressed as $$m=\frac{1}{16\pi} \lim\limits_{R\rightarrow +\infty} \int\limits_{\{|x|=R\}} (\Ricc(g)-\frac{1}{2}R(g)g)(X_{(0)},\nu)\; d\mu_g,$$ and following \cite{lan:cm}, a center can be defined, when $m\neq 0$ and (\ref{rt}) holds, by 
\begin{eqnarray}
C_I^{\ell}= \frac{1}{16\pi m}  \lim\limits_{R\rightarrow +\infty} \int\limits_{\{|x|=R\}} (\Ricc(g)-\frac{1}{2}R(g)g)(X_{(\ell)},\nu)\; d\mu_g. \label{eq:lancm}
\end{eqnarray}

\subsubsection{The center of mass foliation} In an asymptotically flat chart for $(M,g)$, the coordinate spheres near infinity are approximate solutions to a constant mean curvature (CMC) equation.  One can perturb these to exact solutions either by the mean curvature flow due to Huisken and Yau \cite{hy:cm}, or by an implicit function theorem method due to Ye \cite{ye:cm}, cf. \cite{m:cm}.  The methods produce the same foliation near infinity in the case of positive mass, and this foliation provides a geometric notion of the center of mass.    

In the Huisken-Yau approach, one studies solutions $F^{\sigma}:\mathbb{S}^2 \times I \rightarrow M$ of the flow $\frac{\partial}{\partial t}F^{\sigma}=(h-H)\nu$, with initial condition $F^{\sigma}_0:=F^{\sigma}|_{\mathbb S^2 \times \{ 0\}}$ the standard embedding of the sphere of radius $\sigma$ centered at the origin.  Here $\nu$ is the outward unit normal, $H$ is the inward mean curvature of the surface $F^{\sigma}_t(\mathbb{S}^2)=F^{\sigma}(\mathbb S^2\times \{ t\})=:M^{\sigma}_t$, and $h$ is the integral average of the mean curvature: $\int_{M^{\sigma}_t} (h-H)d\mu_g=0$.  This normalized mean curvature flow  improves the isoperimetric ratio (area decreases while the enclosed volume is fixed).  For $\sigma\geq \sigma_0$ ($\sigma_0$ large), the solution exists for all times $t\geq 0$, the surfaces $F^{\sigma}_t(\mathbb{S}^2)$ converge to surfaces $F^{\sigma}_{\infty}(\mathbb{S}^2):=M^{\sigma}$, and the limiting configuration $\{ M^{\sigma}, \sigma\geq \sigma_0\}$ forms a foliation by stable constant mean curvature spheres.  In fact, we have the following result from \cite{hy:cm}.

\begin{Theorem} \label{hy} 
Consider an asymptotically flat three-manifold $(M,g)$, and suppose in an exterior region the metric can be written $g_{ij}(x)=(1+\frac{m}{2|x|})^4\delta_{ij}+O(|x|^{-2})$ with $m>0$.  There is a $\sigma_0>0$, positive constants $C_1$ and $C_2$, and a vector $C_{HY}\in \mathbb{R}^3$ so that for each $\sigma\geq \sigma_0$, the following are true.  The initial value problem has a unique smooth solution for all times $t\geq 0$.  The surfaces $M^{\sigma}_t=F^{\sigma}_t(\mathbb{S}^2)$ converge exponentially fast to a smooth stable hypersurface $M^{\sigma}$, with constant mean curvature $H_{\sigma}$.  The radial coordinate $r=|x|$ satisfies $|r-\sigma|\leq C_1$ on $M^{\sigma}$, and $|H_{\sigma}-\frac{2}{\sigma}+\frac{4m}{\sigma^2}|\leq C_2\sigma^{-3}$.  The hypersurfaces $M^{\sigma}$ have a joint center of mass vector $C_{HY}$, in the following sense (where $d\mu_{\sigma}$ is the pullback of the Euclidean surface measure on $M^{\sigma}$): \[C_{HY}=\lim\limits_{\sigma\rightarrow +\infty} \frac{\int\limits_{M^{\sigma}} x \; d\mu_e}{\int\limits_{M^{\sigma}} d\mu_e}=\lim\limits_{\sigma\rightarrow +\infty} \frac{\int\limits_{\mathbb{S}^2} F^{\sigma}_{\infty}\; d\mu_{\sigma}}{\int\limits_{\mathbb{S}^2} d\mu_{\sigma}}.\]  \end{Theorem}

Huisken and Yau argue that the foliation is asymptotically round on approach to infinity, and they obtain a certain uniqueness result for the foliation.  Ye likewise obtains asymptotic estimates of the geometry of the leaves of the foliation, and a uniqueness theorem.  We state following result, which follows directly from the approach of Ye \cite{ye:cm}; in particular, notice how the CMC foliation is produced by solving for a shift $\tau$ and a normal perturbation $\varphi$ of large coordinate spheres.
\begin{Theorem}\label{ye} Consider an asymptotically flat three-manifold with an exterior region which satisfies the conditions in Theorem \ref{hy}.  Let $0<\alpha<1$.  There is a $\sigma_0>0$ large enough, and a constant $C>0$ so that for $\sigma\geq \sigma_0$, $M^{\sigma}$ is the image of the embedding $\Phi_{\rho}:\mathbb{S}^2\ni \omega\mapsto \rho\left(\tau(\rho)+ \omega + \varphi(\rho,\omega)\nu(\omega)\right)$, where $\frac{2}{\rho}-\frac{4m}{\rho^2}=H_{\sigma}$, $\|\varphi(\rho,\cdot)\|_{C^{2,\alpha}(\mathbb{S}^2)}\leq \frac{C}{\rho}$ and $\tau(\rho)\leq \frac{C}{\rho}$.  If $g_{\rho}$ is the induced metric on $M^{\sigma}$, then as $\rho\rightarrow +\infty$, $\rho^{-2}\Phi_{\rho}^*(g_{\rho})$ converges in $C^{1,\alpha}$ to the unit round metric on $\mathbb{S}^2$.  
\end{Theorem}

We note here that the mass is assumed to be positive, but the constraint equations are not imposed, and in particular, no local energy condition is assumed. Actually, Ye's result only requires that the mass parameter be \emph{nonzero}.  Positivity of $m$, however, implies that the leaves of the foliation are \emph{stable} for the isoperimetric problem.  It is an interesting question to what extent the leaves of the foliation solve a more global optimization problem.  The relation of the mass to the geometry of the three-manifold via the isoperimetric problem arises in the work of Bray \cite{br:th} (cf. \cite{bm:ic}, \cite{reu, ipfrw}) on the Penrose Inequality, and more recently on an isoperimetric approach to the definition of mass of isolated systems due to Huisken \cite{hu:miso2}. 

We now have three notions of center: $c$ from the flux integrals defining global charges (\ref{eq:cm}), $C_I$ from (\ref{eq:lancm}), and $C_{HY}$ from the geometric foliation by CMC spheres near infinity.  We now discuss the relation between these. 

\subsubsection{Equivalence of centers}
We first state the Corvino-Wu result, which establishes, in the special case where the metric is sufficiently harmonically flat near infinity, that the center of mass from the ADM formulation agrees with the geometric center $C_{HY}$, a question raised by X. Zhang \cite{z:pmt} (and others).

\begin{Theorem}\label{main1} Consider an asymptotically flat three-manifold with an exterior region that admits an asymptotically flat chart in which $g_{ij}(x)=(1+\frac{m}{2|x|}+\frac{ \beta^1 x^1+\beta^2 x^2 +\beta^3 x^3}{|x|^3})^4\; \delta_{ij}+O(|x|^{-3})$, with $m>0$.  Then in this chart, $C_{HY}^k=\frac{2\beta^k}{m}=c^k$. \end{Theorem}

\begin{proof}[Sketch of the proof] By translation, it suffices to prove that in case $g_{ij}(x)=(1+\frac{m}{2|x|})^4\delta_{ij}+O(|x|^{-3})$ with $m>0$, then $C_{HY}=0$ in this chart. The metric $g_{ij}(x)=(1+\frac{m}{2|x|})^4\delta_{ij}+O(|x|^{-3})$ agrees with Schwarzschild in this chart to one higher order in $|x|^{-1}$ than is generally considered in \cite{hy:cm}.  Hence, the radial coordinate spheres $F^{\sigma}_0(\mathbb{S}^2)=\{|x|=\sigma\}$ in this chart are better approximate solutions of the CMC equation; we have arranged this by design by centering them appropriately.  It suffices, then, to show that the flow moves them sufficiently little so that the center of mass determined by the foliation is the zero vector.  To prove Theorem \ref{main1}, then, we need to estimate carefully how much $F^{\sigma}_t$ varies from $F^{\sigma}_0$, and to keep track of the variation in the induced surface measure on $M^{\sigma}_t$.  Both of these quantities evolve by an equation involving $(H-h)$, as $\partial_t F^{\sigma}= (h-H)\nu$, and $\partial_t (d\mu_g)=H(H-h) d\mu_g$.  Basic estimates from \cite{hy:cm} yield $|\nabla H|^2\leq C \sigma^{-8}$, which can be integrated to prove $|H-h|\leq C \sigma^{-3}$.  This is what one expects for a finite center of mass.  The goal, then, is to show $|H-h|$ has better decay when we choose coordinates for which $g_{ij}=(1+\frac{m}{2|x|})^4\delta_{ij}+O(|x|^{-3})$.  Since we are starting the flow well-centered, we expect the spheres to drift very little, in fact on average not at all.  Indeed, we have the following proposition from \cite{cw:cm}.   

\begin{Proposition} There is a $C>0$ so that for all $\sigma\geq \sigma_0$, and for all $t\geq 0$, 
\begin{eqnarray}
\max\limits_{M^{\sigma}_t} |H-h| \leq C \sigma^{-7/2} e^{-5\sigma^{-3}t/2}.
\end{eqnarray} \label{mint}
\end{Proposition}
We can conclude the theorem by inserting this into the definition of $C_{HY}$, and by using the flow and the initial conditions to estimate $F^{\sigma}_{\infty}$.  \end{proof}

In two recent papers, L.-H. Huang has definitively unified and extended the above notions of center of mass to the general setting of Regge-Teitelboim asymptotics (\ref{rt}), which we denote ``AF-RT."  In \cite{lan:cm}, she shows that $c=C_{I}$, and in the case of strong enough asymptotics for which the Huisken-Yau proof yields a foliation, that $c=C_I=C_{HY}$.  In \cite{lan:cmc} she proceeds obtain a foliation near infinity by CMC spheres under AF-RT asymptotics, essentially unique for $m\neq 0$, and which is stable for $m>0$.  The center of mass of the foliation again agrees with $C_I$.   We state the main existence theorem from \cite{lan:cmc}. 

\begin{Theorem} Assume $(M,g,K)$ is an AF-RT end with decay rate $q\in (\tfrac{1}{2}, 1]$ and center of mass $c$. If $m\neq 0$, then for $R$ sufficiently large, there exist surfaces $\Sigma_R$ of constant mean curvature $H_{\Sigma_R}=\frac{2}{R}+O(R^{-1-q})$.  $\Sigma_R$ is the graph of a function $\psi_0$ over the coordinate sphere $S_R(c)=\{x:|x-c|=R\}$, i.e. $\Sigma_R=\{ x+\psi_0(x)\nu_g: x\in S_R(c)\}$.   If $\psi_0^*(x)=\psi_0(c+Rx)$, then $\|\psi^*_0\|_{C^{2,\alpha}}+R\|(\psi_0^*)^{\text{odd}}\|_{C^{2,\alpha}}\leq c_0 R^{1-q}$.  Thus $c$ is the geometric center of the family $\Sigma_R$.  If $m>0$ the surfaces $\Sigma_R$ form a strictly stable foliation. \end{Theorem}
 
We do not go into the proof, which involves very beautiful applications of elliptic theory, but we do note a connection to the above asymptotics discussion. 
Huang establishes and uses an extension of the density theorem for harmonic asymptotics.  In \cite{cs:ak}, it is shown that solutions with these asymptotics are dense in the space of asymptotically flat solutions, with the energy-momentum four-vector approaching that of the original data.  In \cite{lan:cm}, Huang argues that if the data is AF-RT to start, then the odd part of $g$ and the even part of $K$ can also be made arbitrarily small, and thus so can the change in all global charges under the approximation, cf. \cite{hsw:ang}.  We note that Huang also obtains interesting uniqueness results which one may compare with \cite{hy:cm, qt:cm}.

We will return to the issue of the asymptotics in Section \ref{Asymptotic Gluing}. 

\section{The Conformal Method}
\label{ConfMet}

The most successful approach so far for systematically studying the existence
and uniqueness of solutions to   (\ref{eq:c1})-(\ref{eq:c3}) is
through the conformal method of Lichnerowicz~\cite{Lich44},
Choquet-Bruhat and York~\cite{CBY}. The idea is to introduce a
set of unconstrained \emph{conformal data}, which are freely
chosen, and find  $(\threeg,K)$ by solving a {\it determined} system of partial differential equations. In this section, we treat the general $n$-dimensional case, as the theory is the same in all dimensions $n\geq3$.  In the vacuum case
with vanishing cosmological constant~\cite{CBY}, the free
conformal data consist of an $n$-dimensional manifold $\hyp$, a Riemannian
metric $\tthreeg$ on $\hyp$ (with Levi-Civita connection $\tD$), a trace-free symmetric tensor 
$\tsigma$, and the \emph{mean curvature function} $\tau$. The
initial data $(\threeg,K)$ defined as
\begin{eqnarray}
\threeg & = & \phi^q \tthreeg \label{threegeq}\\
K & = & \phi^{-2} (\tsigma + \tcalD W) + \frac{\tau}{n}
 \phi^q \tthreeg \;, \label{Keq}
\end{eqnarray}
where  $q=\frac{4}{n-2}$ and $\phi$ is a positive function, will then solve \eq{eq:c1}-\eq{eq:c2}
if and only if $\phi$ and the vector field $W$
solve the equations
 \beq {\div}_{\tthreeg}(\tcalD
W + \tsigma)=\frac{n-1}{n}\phi^{q+2} \tD \tau
 \;,
 \label{vacmom:noncmc}
\eeq
\beq \label{vacham:noncmc} \laplaciano{\tthreeg} \phi -
\frac{1}{q(n-1)} R(\tthreeg)\phi + \frac{1}{q(n-1)}|\tsigma + \tcalD
W |^2_{\tthreeg} \phi^{-q-3}-\frac{1}{qn}\tau^2\phi^{q+1}=0\;,  \eeq
where $\tcalD $ is the \emph{conformal Killing operator}:
\beq \tcalD W_{ab} = \tD_a W_b + \tD_b W_a -
\frac{2}{n}\tthreeg_{ab} \tD_c W^c \;.
\label{LWeq}
 \eeq
 Vector fields $W$  annihilated by
$\tcalD$  are  \emph{conformal Killing vector fields}, and are
characterized by the fact that they generate (perhaps local)
conformal diffeomorphisms of $(\hyp, \threeg=\phi^q \tthreeg)$.  The semi-linear scalar
equation (\ref{vacham:noncmc}) is often referred to as the
\emph{Lichnerowicz equation}.

Equations (\ref{vacmom:noncmc})-(\ref{vacham:noncmc}) form a
determined system of equations for the scalar-vector pair $(\phi,
W)$. The operator $\div_{\tthreeg}(\tcalD \;\cdot)$ is a linear,
formally self-adjoint, elliptic operator on vector fields. What
makes the study of the system
(\ref{vacmom:noncmc})-(\ref{vacham:noncmc}) difficult in general is
the nonlinear coupling between the two equations.

The explicit choice of \eq{threegeq}-\eq{Keq} is motivated by the
two identities 
\beq R({\tilde \threeg}) = -\phi^{-q-1} (q(n-1)\Delta_\threeg \phi -
R({\threeg})\phi)
 \;,
  \label{ConfReq} \eeq
where $\tilde \threeg=\phi^{-q} \threeg$, and
\beq \tD^a (\phi^{-2} B_{ab}) = \phi^{-q-2} D^a
B_{ab} \label{confdiveq} \eeq
which holds for any trace-free
tensor $B$.  Equation (\ref{ConfReq}) is the well known identity
relating the scalar curvatures of two conformally related
metrics, and we note that $q=\frac{4}{n-2}$ is the unique exponent that
does not lead to supplementary $|D \phi|^2$ terms in
(\ref{ConfReq}).

In the space-time evolution $(\mathcal M^{n+1}, \fourg)$ of the initial
data set $(\hyp, \threeg,  K)$,
the function $\tau= \tr_\threeg K $ is the mean curvature of the hypersurface $\hyp\subset\mathcal M$.  The
assumption that the mean curvature function $\tau$ is
 constant on $\hyp$ significantly simplifies the analysis of the vacuum constraint
equations, because it decouples equations  (\ref{vacmom:noncmc}) and
(\ref{vacham:noncmc}). One can then attempt to solve
(\ref{vacmom:noncmc}) for  $W$, and then solve the Lichnerowicz
equation (\ref{vacham:noncmc}).

\subsection{Existence via the conformal method: CMC initial data}

Existence and uniqueness of solutions of this problem for constant
mean curvature (CMC) data  has been studied extensively. For
compact manifolds this was exhaustively analysed by
Isenberg~\cite{Jimconstraints}, building upon a large amount of
previous work~\cite{Lich44,OMYork73,York74,CBY}; the proof was
simplified  by Maxwell in~\cite{Maxwell:compact}. If we let
\beq
\label{YamInv}
{\mathcal Y}([\threeg])=\inf_{f\in C^{\infty}(M),\\ f\not\equiv 0}\frac{\int_M (|\nabla f|^2 + 
\frac{1}{q(n-1)} R(\threeg) f^2)dv_{\threeg}}{\|f\|^2_{L^{2^*}}},
\eeq
where $2^*=2n/(n-2)$, 
denote the Yamabe invariant of the
conformal class $[\threeg]$ of metrics determined by $\threeg$
(see~\cite{lp:yam}), the result reads as follows:
\begin{Theorem}[\cite{Jimconstraints}]
\label{TJimC}
 Consider a smooth conformal initial data set
$(\tthreeg, \tsigma, \tau)$ on a compact manifold $\hyp$, with
constant $\tau$. Then there always exists a solution $W$ of
\eq{vacmom:noncmc}. Setting $\sigma= \tcalD W + \tsigma$,  
 a positive solution $\phi$ of the Lichnerowicz equation exists if and only if one of the following conditions holds
 \begin{itemize}
 \item[1.] ${\mathcal Y}([\tthreeg])>0, \sigma \not\equiv 0$.
 \item[2.] ${\mathcal Y}([\tthreeg])=0, \sigma \not\equiv 0, \tau\not=0$.
 \item[3.] ${\mathcal Y}([\tthreeg])<0, \tau\not=0$.
 \item[4.] ${\mathcal Y}([\tthreeg])=0, \sigma \equiv 0, \tau=0$.
 \end{itemize}
 Moreover, we have uniqueness in every case except case 4 (where one may homothetically scale the metric to choose the volume arbitrarily).
\end{Theorem}

Beyond a simple existence result, we see that Theorem  \ref{TJimC} does more. It actually solves 
the {\it Conformal Parametrization Problem} for CMC initial data on a compact manifold. In particular, given a smooth, compact Riemannian manifold $(\hyp, \tthreeg)$, it provides a full parametrization of the set of vacuum solutions $(\threeg, K)$ to the Einstein constraint equations such that $\threeg$ lies in the conformal class $[\tthreeg]$ of $\tthreeg$.  This parametrization is given by the choice of a constant $\tau$ and a trace-free symmetric tensor $\tsigma$ such that one of the four conditions of  Theorem  \ref{TJimC} are satisfied (note that 
$\sigma$ vanishes if and only if  $\tsigma$ vanishes).  The extent to which the Conformal Parametrization Problem may be solved for  non-CMC initial data is unclear, and is currently actively being explored, as we discuss below.

\begin{proof}[Remarks on the proof of Theorem \ref{TJimC}]
It is worth commenting on the role of the Yamabe invariant in this result.  Indeed the Lichnerowicz equation bears a great deal of similarity to the Yamabe equation specifying the existence of a conformal metric $h=\phi^q\tthreeg$ of constant salar curvature $R(h)=C$
\beq \label{Yamabe} 
\laplaciano{\tthreeg} \phi -
\frac{1}{q(n-1)} R(\tthreeg)\phi + \frac{1}{q(n-1)}C\phi^{q+1}=0\;. \eeq
There are two key features which distinguish the two equations: one is the existence of the term involving $\phi^{-q-3}$ in the Lichnerowicz equation and the other is the sign of the term involving $\phi^{q+1}$.  These distinctions are crucial.  Indeed the solution of the Yamabe problem \cite{Schoen:Yamabe, lp:yam} was one of Richard Schoen's significant contributions to geometric analysis and a very important part of late twentieth century mathematics.  The original proof of Theorem \ref{TJimC} used the solution of the Yamabe problem in full.  Maxwell showed in~\cite{Maxwell:compact} that the proof in fact relies only on the existence of a metric of within a given conformal class whose scalar curvature has the same sign as that of the Yamabe invariant  ${\mathcal Y}([\threeg])$.  This far simpler fact is expressed in the following well known and very useful result.
\begin{Proposition}\label{tfae}  Let $L_{\tthreeg}=-\laplaciano{\tthreeg} +
\frac{1}{q(n-1)} R(\tthreeg)$ be the conformal Laplacian, and let 
$\mu_{\tthreeg}$ be its first eigenvalue.  Then the following are equivalent.
\begin{itemize}
 \item[1.] ${\mathcal Y}([\tthreeg])>0$ (respectively $=0$ or $<0$).
 \item[2.] $\mu_{\tthreeg} >0$ (respectively $=0$ or $<0$).
 \item[3.] There exists a positive function $u$ on $M$ such that the scalar curvature of $u^q\tthreeg$ satisfies $R(u^q\tthreeg)>0$  (respectively $=0$ or $<0$) everywhere on $M$.
 \end{itemize}
\end{Proposition}
Maxwell established this result in a low regularity setting in~\cite{Maxwell:compact} in order to carry out a program for analysing the conformal method for solving the constraint  equations for metrics of low
differentiability (see also \cite{Choquet-Bruhat:safari}).  This
was motivated in part by work  on the evolution problem for
\emph{rough initial data}~\cite{KlainermanRodnianski:r1,KlainermanRodnianski:r2,KlainermanRodnianski:r3,SmithTataru:sharp}.
Judicious applications of this result, together with the maximum principle, allow one to 
prove Theorem \ref{TJimC} using the method of sub and super-solutions.
\end{proof}

The conformal method easily extends to the CMC constraint
equations for many non-vacuum initial data sets, e.g. the
Einstein-Maxwell system~\cite{Jimconstraints}, where one obtains
results very similar to those of Theorem~\ref{TJimC} (see also \cite{IMaxP}). However,
other important examples, such as the Einstein-scalar field
system~\cite{CBIP1,CBIP2,CBIP-Leray,HPP} are not as well understood.

\subsection{Existence via the conformal method: near-CMC and far-from-CMC initial data}

Conformal data which is close to being CMC  (e.g. as measured by a smallness
assumption on $|\nabla\tau|/|\tau|$) are usually referred to as
``near-CMC."  Classes of near-CMC conformal data solutions have
been constructed~\cite{IM,CBIM91,ICA07,VinceJim:noncmc}.  
In particular, with an appropriate definition of near-CMC, it is know that on a Riemannian manifold with no conformal Killing fields, the conditions of Theorem \ref{TJimC} yield a unique solution of the constraint equations provided that the mean curvature is nowhere vanishing.
There is at least one example of a non-existence theorem
\cite{IsenbergOMurchadha} for a class of near-CMC conformal
data. 
However, due to the nonlinear coupling in the system
(\ref{vacmom:noncmc})-(\ref{vacham:noncmc}),  the question of
existence for unrestricted choices of the mean curvature $\tau$
appears to be significantly more difficult, and until recently
all results assumed strong restrictions on the gradient of
$\tau$. 

Alan Rendall has demonstrated that on $\mathbb S^2\times \mathbb S^1$ endowed with the product metric, the data consisting of $\tsigma\equiv 0$ and $\tau$ equal to an (arbitrary) odd function on $\mathbb S^1$ has no solution which shares the symmetry of the conformal data.  Therefore either there are no solutions, or there are more than one solution. 

The first general existence result in the ``far-from-CMC" context  is due to
Holst, Nagy, and Tsogtgerel~\cite{HNT07,HNT08}.  In the Yamabe positive (${\mathcal Y}([\tthreeg])>0$) case, they construct solutions with freely specified mean curvature; however they assume both the presence
of (sufficiently weak) matter fields and that $\tsigma$ is pointwise sufficiently small (depending on $\tau$) and not identically zero.  One may view the smallness assumption on $\tsigma$ as standing in for the near-CMC hypothesis in previous work. In~\cite{Maxwell08}, Maxwell provides a sufficient
condition, with no restrictions on the mean curvature, for the
conformal method to generate solutions to the vacuum constraint
equations on compact manifolds.  As an application, Maxwell
demonstrates the existence of a large class of solutions to the
\emph{vacuum} constraint equations with freely specified mean
curvature (again assuming that $\tsigma$  is pointwise sufficiently small and not identically zero).
In the results of Holst, Nagy, and Tsogtgerel, as well as those of Maxwell, the methods employed do not allow one to assert the uniqueness of the solutions found.  
Nonetheless, these results together represent a significant
advance in our understanding of how the conformal method may be
used to generate solutions of the vacuum constraint equations.
However the existence question for generic classes of large
conformal data remains wide open. 

In order to explore the ``large data" regime more systematically, Maxwell has recently studied
a model problem in the Yamabe null (${\mathcal Y}([\tthreeg])=0$) class \cite{Maxwell09}.
By considering a three-parameter family of model conformal data that allow for simultaneous violations of both the near-CMC and small-$\tsigma$ conditions, Maxwell was able to identify a number of new phenomena.  First, he confirmed that for this class of data there is also a small-$\tsigma$ result.  He was also able to assert the non-existence of solutions when the data violated both the near-CMC and small-$\tsigma$ conditions.  Moreover, he was able to demonstrate that the small-$\tsigma$ solutions are not unique.  These results indicate that the landscape for studying the constraint equations via the conformal method will  be quite interesting when one moves away from small data as represented by either the near-CMC or small-$\tsigma$ conditions.  It seems likely that in order to effectively analyse this situation, a new approach to studying the constraint equations may be needed.

In~\cite{BartnikIsenberg} the reader will find a presentation
of alternative approaches to constructing solutions of the
constraints, covering work done up to 2003.

\subsection{The constraint equations on asymptotically
flat manifolds}
 \label{ssceafm} 
There are a large number of
well-established results concerning the existence of CMC and
near-CMC solutions of the Einstein constraint equations on
asymptotically flat manifolds \cite{Cantor77, ChIY, CBY,
CBChristodoulou, YCB:GRbook, Maxwell:AH, Maxwell:rough} using the conformal method. 
In 1977, Cantor introduced a quantity, analogous to the Yamabe invariant, which was  sufficient  to prove the existence of a positive solution to the Lichnerowicz
equation relative to a given set of asymptotically flat CMC conformal data~\cite{Cantor77}. Since in the asymptotically flat setting CMC means that $\tau=0$, the relevant Lichnerowicz equation is simply
\beq \label{afLich} \laplaciano{\threeg} \phi -
\frac{1}{q(n-1)} R(\threeg)\phi + \frac{1}{q(n-1)}|\sigma|^2_{\threeg} \phi^{-q-3}=0\;. 
\eeq
Maxwell~\cite{Maxwell:AH} uncovered an error in
Cantor's definition of the invariant~\cite{Cantor77} and
provided the correct definition.  For an asymptotically flat manifold $(M, \threeg)$, we define
\beq
{\mathcal Y}_{AF}([\threeg])=\inf_{f\in C^{\infty}_c(M),\\ f\not\equiv 0}\frac{\int_M (|\nabla f|^2 + 
\frac{1}{q(n-1)} R(\threeg) f^2)dv_{\threeg}}{\|f\|^2_{L^{2^*}}},
\eeq
(compare with (\ref{YamInv})).
The precise result is then the following. 
\begin{theorem} Suppose that $(M, \threeg)$ is asymptotically flat of class $W^{k, 2}_{-\tau}$
for $k>n/2$ and $\tau\in(0, n-2)$ and that $\sigma\in W^{k-1, 2}_{-\tau-1}$ is a transverse-traceless tensor.  Then there exists a positive conformal factor $\phi$ satisfying 
(\ref{afLich}) if and only if ${\mathcal Y}_{AF}([\threeg])>0$. Moreover, if a solution exists, then it is unique.  
\end{theorem}
The key step to establishing this is to have an appropriate analog of Proposition \ref{tfae}.  This is given by the following result.
\begin{Proposition}Suppose that $(M, \threeg)$ is asymptotically flat of class $W^{k, 2}_{-\tau}$
for $k>n/2$ and $\tau\in(0, n-2)$.  Then the following are equivalent.
\begin{itemize}
\item[1.] There exists a conformal factor $\phi>0$ such that $1-\phi\in W^{k, 2}_{-\tau}(M)$ and such that $\phi^{q}\threeg$ is scalar flat.
\item[2.] ${\mathcal Y}_{AF}([\threeg])>0$.
\item[3.] For each $\eta\in [0,1]$, ${\mathcal P} _{\eta} = - \laplaciano{\threeg} + \eta
\frac{1}{q(n-1)} R(\threeg)$ is an isomorphism acting on $W^{k, 2}_{-\tau}(M)$.
\end{itemize}
\end{Proposition}

\section{Gluing Constructions}
 
Gluing constructions, by which known solutions of a geometric partial differential equation are combined to produce new solutions, are now ubiquitous in geometric analysis.  One of the earliest gluing results concerns scalar curvature.  In 1979, Schoen and Yau \cite{sy:manu} showed that, within the category of compact $n$-dimensional manifolds, the property of admitting a metric with strictly positive scalar curvature is preserved under surgeries of codimension $k\geq 3$.  
In particular, this includes taking connected sums, which can be viewed as codimension $n$ surgery.
(Gromov and Lawson \cite{GL1, GL2}  independently established these results, introducing other important techniques into the study of manifolds with positive scalar curvature.)  
 Over the past ten years, there have been many applications of gluing constructions to general relativity.  All of these take place at the level of the constraint equations and have implications for space-times by considering the evolution of the initial data sets constructed by gluing.  

\subsection{Asymptotic gluing} \label{Asymptotic Gluing}
It was an interesting open question for many years to what extent the asymptotic expansion at infinity determines the interior behavior of a solution of the (vacuum) constraint equations, cf. \cite[p. 371]{sy:dg}, \cite{ba:op}.  The simplest such question was answered in the resolution of the Positive Mass Theorem: if near infinity the solution is precisely Euclidean, then the solution is globally flat.  A natural next question to pose, then, is whether an asymptotically flat metric of vanishing scalar curvature, which near infinity agrees \emph{precisely} with a standard time-symmetric slice of a Schwarzschild solution, must be globally Schwarzschild.  The question is to some extent about unique continuation under the constraint of vanishing scalar curvature, inspired possibly in part due to the success of the conformal method discussed above, which turns the constraints, which are \emph{undetermined}-elliptic, into a determined elliptic system.  We remark that it was known that unique continuation does not hold with respect to the interior, as Bartnik \cite{ba:QS} constructed non-trivial scalar-flat metrics on $\mathbb R^3$ with a flat sub-domain.  We also note that, as we recalled above in Proposition \ref{br:sa}, Bray showed that there are plenty of metrics which are Schwarzschild near infinity and have non-negative scalar curvature, and in particular, Cutler and Wald \cite{cw:em} showed that there exist solutions of the Einstein-Maxwell constraints on $\mathbb R^3$ which near infinity agree with a Schwarzschild metric.  We do remark, however, that the Schwarzschild metric satisfies a rigidity condition, with respect to the \emph{Penrose inequality} \cite{br:pen, hi:pen}.  

In the late nineties, Corvino and Schoen resolved the question in a very strong form, as we discuss below.  The main idea which echoes in their work is that there is a lot of freedom in the initial data, and one ought to look outside the space of conformal deformations.  The idea that took hold was a localized version of the Fischer-Marsden results \cite{fm:def, fm:defcon}, the possibility for which was in part inspired from the work of Lohkamp \cite{loh:ham}. 
It was shown that not only do there exist many initial data sets for the vacuum constraints which agree with Schwarzschild (time-symmetric) or Kerr near infinity, but furthermore such solutions are \emph{dense} in an appropriate topology. 
 
We now state the time-symmetric result from \cite{cor:schw}. 

\begin{Theorem}  Let $(M,g)$ be an asymptotically flat three-manifold with zero scalar curvature.  Let $E\subset M$ be any asymptotically flat end, and let $E_{r_0}$ be an exterior region in $E$ corresponding to $\{ x: |x| > r_0\}$ in asymptotically flat coordinates.  Let $k$ be a non-negative integer.  Then for any $\epsilon>0$, there is an $R>0$ and a (smooth) metric $\bar{g}$ with zero scalar curvature and $\|g-\bar{g}\|_{C^k(E)}<\epsilon$ (the norm is taken with respect to the Euclidean metric in the asymptotically flat coordinate chart), so that $\bar{g}$ is equal to $g$ on $M\setminus E_{R}$, and $\bar{g}$ is identical to an asymptotically flat end of a standard Schwarzschild slice on $E_{2R}$.  The analogous statement holds for $n\geq 3$.  \label{thm:schw}
\end{Theorem}

The general case of the constraints was addressed in \cite{cs:ak}, cf. \cite{cd}.   We define a family of solutions on the exterior of a fixed ball and smoothly parametrized on an open set $\mathcal{O}\subset \mathbb{R}^{10}$ to be \textit{admissible} if, with reference to a fixed coordinate chart near infinity, the family satisfies (\ref{rt}) locally uniformly, and the map $\Theta:\mathcal{O}\rightarrow \mathbb{R}^{10}$ which associates to each member of the family its energy-momenta $(m,P,J,mc)$ is a homeomorphism onto an open subset.  We note that slices in Kerr form an admissible family, parametrized by the total mass $m$, the angular momentum parameter $a$, and an element in the Poincar\'{e} group to represent Euclidean motions of the asymptotic coordinate system as well as boosts, cf. \cite{cd}.

\begin{Theorem}\label{main2} Let $(M,g,\pi)$ be any asymptotically flat solution of the vacuum constraints.  Given any $\epsilon >0$, there is a solution $(\overline{g},\overline{\pi})$ within $\epsilon$ of $(g,\pi)$ (in a weighted norm) and whose ADM energy-momentum $(E,P)$ is within $\epsilon$ of that of $(g,\pi)$, so that near infinity, $(\overline{g}, \overline{\pi})$ agrees with a member of an admissible asymptotic model family, for example a boosted space-like slice in Kerr.
\end{Theorem} 

We now discuss the main ideas of the proof.  In a given asymptotically flat end, we pick an asymptotically flat coordinate chart, and at a large coordinate radius $|x|=R$, we patch our given data to that of the model using a cutoff function in the annulus $A_R$ from $|x|=R$ to $|x|=2R$.  Since the data is approaching the flat data, the gluing produces an approximate solution of the constraints.  At this point we seek to perturb this to an exact solution.  This is often done using the conformal method; however, the conformal factor induces a global change, albeit a small one far away from the gluing region, and we seek to keep the data \emph{unchanged} outside the annulus $A_R$.  The key that enables one to do this is the underdetermined nature of the constraints, i.e. the overdetermined-ellipticity of the adjoint of the linearized constraint operator.  In fact, the construction relies heavily on linear elliptic estimates for this adjoint, as we now discuss.  

\subsubsection{Localized scalar curvature deformation}
Motivated by the preceding discussion, we study the following \emph{localized scalar curvature deformation problem}: Given a smooth domain $\Omega\subset (M,g)$ and a compactly contained sub-domain $\Omega_0$, for $f$ supported in $\Omega_0$ and sufficiently small, find a tensor $h$ supported in $\overline{\Omega}$ so that $R(g+h)=R(g)+f$, with $h$ small depending on $f$.  A natural way to approach this, in the spirit of Fischer-Marsden \cite{fm:def}, is to study the linearization of the scalar curvature operator. 

Let $L_g$ be the linearized scalar curvature operator: $L_g (h) =\frac{d}{dt}\big|_{t=0} R(g+th)$.  Then we have $L_g(h)=-\Delta_g (\tr_g h)+\Div_g(\Div_g h)-h\cdot \Ricc(g)$, and thus $L_g^* u = -(\Delta_g u ) g+ \Hess_g u -u \Ricc(g)$.  By taking the trace of this equation, we see $\tr_g(L_g^*u)=-(n-1) \Delta_g u - u R(g)$, so that $\Delta_g u$, and hence the full Hessian, is controlled pointwise by $L_g^* u$ and $u$.  So we get the following estimate, valid on any domain (\emph{without imposing boundary conditions}): 
\begin{equation} \label{be}
\|u\|_{H^2(\Omega)} \leq C \left( \|L_g^*u\|_{L^2(\Omega)} + \|u\|_{H^1(\Omega)}\right). 
\end{equation}
We can replace the lower-order $H^1$-term by the $L^2$-norm on smooth bounded domains follows by using interpolation on the compact inclusions $H^2\hookrightarrow H^1 \hookrightarrow L^2$.  In fact if we restrict $u$ to lie in a subspace transverse to the kernel of $L_g^*$ (generically there is no kernel, but we will be interested in the case where $g$ is the flat metric, which has a four-dimensional kernel), then we have an estimate with no lower-order term at all.  

This estimate can be used to get analogous estimates in weighted spaces.  Indeed, let $\rho$ be a smooth, positive function on $\Omega$, with $\rho(x)=(d(x))^N$ near $\partial \Omega$, where $d=d(x)$ is distance of $x$ to $\partial \Omega$, and $N$ will be taken to be sufficiently large (one may also take $\rho(x)=e^{-1/d(x)}$ near $\partial \Omega$); we take $\rho$ to depend only on $d$ and to be monotonic in $d$, leveling off to a positive constant away from $\partial \Omega$.  From the above, choosing $u$ to lie in a fixed subspace transverse to the kernel of $L_g^*$ (if it is non-trivial), we get an estimate $\|u\|_{H^2(\Omega_{\epsilon})} \leq C \|L_g^*u\|_{L^2(\Omega_{\epsilon})}$, for small $\epsilon\geq 0$, where $\Omega_{\epsilon}=\{ x\in \Omega: d(x)> \epsilon\}$ and $C$ is uniform in $\epsilon$. For example, such a uniform estimate holds for metrics near the flat metric $g_E$, for functions $u$ transverse to the kernel $K_0=\mbox{span}\{1, x^1, x^2, x^3\}$.  By multiplying the square of this estimate by $\rho'(\epsilon)$ and integrating by parts, we get the weighed estimate $\|u\|_{H^2_{\rho}(\Omega)} \leq C' \|L_g^*u\|_{L^2_{\rho}(\Omega)}$.  There are similar weighted estimates for $D\Phi^*$, which are slightly harder to prove, but hold nonetheless. 

We illustrate the usefulness of such weighted estimates by presenting a simpler example which W. Qiu \cite{q:cag} adapted from the analysis in \cite{cor:schw}.  Here we look at the divergence operator, whose adjoint is $-\nabla$, say at the flat metric.  Suppose we want to solve $\Div X = f$ on an open ball $B$.  Suppose that $f$ has compact support in $B$, and that we want $X$ to decay to zero at $\partial B$.  Of course a necessary condition on $f$ is that its integral over the ball vanish, $\int_B f\; dv_e=0$, i.e., $f$ must be orthogonal to the kernel of the adjoint of $\Div$.  We solve this problem variationally: let $\mathcal{F}(u):= \int_B \left( \frac{1}{2} |\nabla u|^2 \rho -  uf \right)dv_e$, where $\rho$ is as above; for the problem at hand, we could follow \cite{q:cag} and also choose $\rho$ to be identically 1 on the support of $f$, but this is not essential.  We minimize $\mathcal F (u)$ over all $u\in H^1_{loc}(B)$ so that $u^2\rho$ and $|\nabla u|^2 \rho$ are integrable (these integrals can be used to define the (squares of) the weighted $L^2_{\rho}$ and $H^1_{\rho}$-norms), and so that $u$ is $L^2$-orthogonal to $\zeta$, where $\zeta\geq 0$ is a smooth bump function with compact support.   Since $u$ is taken to lie in a space transverse to the kernel of $\nabla$, we have an estimate (as above) of the form \[\|u\|_{H^1_{\rho}(B)}\leq C \|\nabla u\|_{L^2_{\rho}(B)}\; ,\]  which is simply a weighted Poincar\'{e} inequality, cf. \cite{q:cag}.  Thus \[\mathcal{F}(u)\geq \frac{1}{2C^2} \|u\|^2_{H^1_{\rho}(B)}- \|u\|_{L^2_{\rho}(B)}\|f\|_{L^2_{\rho^{-1}}(B)}.\]  Standard functional analysis yields a minimizer $u$.   The Euler-Lagrange equation is simple to compute, and indeed letting $X=-\rho \nabla u$ we see that $\Div X =f + \lambda \zeta$ for some $\lambda\in \mathbb{R}$.  By the elliptic regularity, we see that $u$ is smooth in the interior of $B$, and moreover, by integrating the Euler-Lagrange equation against the constant function 1, we see that $\lambda=0$ and so $X$ solves the original equation (the integration by parts is justified by the pointwise estimate below).  Now although $u$ might not decay at $\partial B$, we have that $X$ does, by the Schauder estimates. Indeed, let $x\in B$ have distance $d$ to $\partial B$, and let $B_1$ and $B_2$ be the balls of radius $\frac{d}{3}$ and $\frac{2d}{3}$ about $x$; we can take $d<1$ so small that $B_2$ is outside the support of $f$.  We have the elliptic equation $\Delta u +\frac{\nabla \rho}{\rho} \cdot \nabla u= -\frac{f}{\rho}$.  We take $\rho$ equal to a large power $N$ of the distance to the boundary, near the boundary.  Then by standard elliptic estimates, we have, with $\psi=\psi(n, k, \alpha)$ and using $f=0$ on $B_2$, and $\int_{B_2}u^2 dv_e \leq Cd^{-N} \int_{B_2} u^2 \rho\; dv_e$,
\begin{eqnarray*}
\|X\|_{C^{k,\alpha}(B_1)}= \|\rho \nabla u\|_{C^{k,\alpha}(B_1)} & \leq & C d^{N-\psi} \left(\| f\rho^{-1}\|_{C^{k-1,\alpha}(B_2)} + \|u\|_{L^2(B_2)}\right) \\ 
&\leq & Cd^{\frac{N}{2}-\psi} \|u\|_{L^2_{\rho}(B)} \leq  Cd^{\frac{N}{2}-\psi} \|f\|_{L^2_{\rho^{-1}}(B)} .
\end{eqnarray*}
The last inequality follows from the inequality $\mathcal{F}(u)\leq 0$ and the weighted estimate.  From this, we see that we can solve for $X$ decaying at the boundary along with as many derivatives as we like, provided we choose $N$ large enough.  

We now return to the scalar curvature problem.  Suppose that $L_g^*$ has trivial kernel on $\Omega$.  Essentially the same analysis, using the weighed estimates and the functional $\mathcal F (u)= \int\limits_{\Omega} (\frac{1}{2} \rho |L_g^*u|^2 - uf) \; dv_g$, allows use to solve $L_g(\rho L_g^* u)=f$ variationally.  For instance, if $f$ is small, then $h=\rho L_g^*u$ will be small enough so that $g+h:=g+\rho L_g^*u$ is a metric, and we have $R(g+h)= R(g)+f + Q(h)$. 
Thus we see that in case $L_g^*$ has trivial kernel, then we can choose $f$ small and solve for the localized scalar curvature deformation, at least to first order.  We can iterate linear corrections to solve the non-linear problem, i.e. we can solve the localized scalar curvature deformation problem as posed above.  On the other hand, metrics like the Euclidean metric, for which $L_g^*$ has non-trivial kernel, are called \emph{static}, and elements in the kernel are sometimes called \emph{static KIDs}.  Static metrics are very special; for instance they must have constant scalar curvature.  This local deformation result cannot hold at static metrics; for example in connection with the Positive Mass Theorem,  the flat three-torus does not admit any positive scalar curvature metric, and analogously, there is no compactly supported deformation of the Euclidean metric with non-negative scalar and non-constant scalar curvature.  The round sphere is also static, and it was only recently that Brendle, Marques and Neves provided a counterexample to the \emph{Min-Oo conjecture}: they construct metrics on $\mathbb S^n$, $n\geq 3$, which have $R(g)\geq n(n-1)$, \emph{with strict inequality holding at some point},  with $g$ a unit round metric in a neighborhood of a hemisphere \cite{min-oo}. 

\subsubsection{Localized deformations, global charges, and the proofs of Theorems \ref{thm:schw} and \ref{main2}}

To prove Theorems \ref{thm:schw} and \ref{main2}, we patch together two solutions of the constraints in annulus $A_R$ near infinity, producing an approximate solution.  The data there is close to $(g_E,0)$, so to get uniform estimates on the adjoint of the linearized constraint operator, we have to work transverse to $K_0=\mbox{ker } L^*$, respectively $K=\mbox{ker } D\Phi^*$.  Indeed, using the methods discussed above along with suitable projections, we can, at the linear level, solve the constraints up to the kernel $K_0$, respectively $K$.  By iterating the linear solvability in a Newton-Picard iteration, we can produce $\tilde{g}$ (respectively $(\tilde g, \tilde{\pi})$) which solves $R(\tilde{g})\in \zeta K_0$, respectively $\Phi(\tilde{g},\tilde{\pi})\in \zeta K$ (for $\zeta$ a suitable bump function), and which agrees with the glued approximate solution except in the gluing region $A_R$; that is, it is exactly the model solution on $E_{2R}$, exactly the original on $M\setminus E_R$, and a perturbation of the glued solution in $A_R$.  The key is that we have used the overdetermined-elliptic estimates to solve for a compactly supported (in $\overline{A_R}$) perturbation.  

So we now have our constraint operator with values in $\zeta K$; of course we want this value to be zero.  Naturally (as with the divergence example above), we take our constraint values and integrate against $K$ to see what we get.  For example, for the kernel element $1\in \text{ker } L^*$, we get (compare (\ref{bdyint}))
\begin{align*} 
 \int_{A_R} R(\tilde{g})\; dv_e \approx \int_{A_R} \sum_{i,j}(\tilde{g}_{ij,ij}-\tilde{g}_{ii,jj})\; dv_e &=  \int_{\partial A_R} \sum_{i}(\tilde{g}_{ij,i}-\tilde{g}_{ii,j}) \nu_e^j d\mu_e\\ &\approx 16\pi \Delta m,
 \end{align*}
where $\Delta m$ is the mass of the outer Schwarzschild minus the mass of the original data.  Analogously, integration against a linear function $x^k\in K_0$ yields the change $16\pi \Delta(mc)$. 
For the Killing field $X=\frac{\partial}{\partial x^1}\in \text{ker } D\Phi^*$, we have $$\int_{A_R} (\Div_{\tilde g}\tilde{\pi})_i X^i \; dv_e\approx \int_{A_R} (\Div_{g_E}\tilde{\pi})_i X^i \; dv_e= \int_{\partial A_R} \tilde{\pi}_{ij}X^i\nu_e^j \; d\mu_e\approx 8\pi \Delta P_1.$$  Integrating against a rotation field gives the change in angular momentum.
 
Thus we see that the constraint equations imply an asymptotic conservation law across the boundary of the annulus.  The conserved quantity is exactly the boundary integral coming from the divergence theorem applied to the inner product of the constraint operator and the generator of the asymptotic symmetry, i.e. the element of the kernel of the adjoint of the linearized constraint operator.  For our gluing procedure, we see that our model near infinity has to roughly preserve the conserved quantities coming from the kernel of the linearized constraints.  It is this kernel, then, that dictates which part of the asymptotics has to be preserved in our construction, and so our model solutions have to constitute a family which exhibits all possible values of the asymptotic quantities in an effective way.  Thus far, our gluing and perturbation procedure depended only on the asymptotic flatness, and so we were not quite able to solve the constraints.  We now pay attention to which one of the model family we have glued on, and show that one of these works to solve the constraints.  
 
In the time-symmetric vacuum case, we are trying to make the metric $\tilde{g}$ be scalar-flat.  
Of course, what we actually have is a \emph{family} of metrics, one for each of the Schwarzshild family $( 1+ \frac{m}{2|x-c|})^4 g_E$ we glue on, parametrized by the mass $m$ and center of mass $c$ (relative to the fixed asymptotic chart).  Our procedure, then, produces a map from the parameter values to the values $R(\tilde{g})\in \zeta K_0$, a map between four-dimensional Euclidean spaces.  The map roughly looks like (up to scaling and constant factors) \begin{eqnarray} \label{pmap0} (m,c)\mapsto \left( m-m_0, mc-m_0c_0 \right)\; ,\end{eqnarray} where $m_0$ and $c_0$ are the mass and center of the original metric; note we can always translate coordinates so that $c_0=0$.  This is the leading term, and the error term goes to zero as $R\rightarrow +\infty$.  For large enough $R$, then, the parameter map is a small perturbation of a local diffeomorphism which covers the origin at the initial mass and center, and thus by degree theory, the actual map hits zero for some $(m,c)$ near $(m_0,c_0)$.  For these parameter values, $R(\tilde{g})=0$.  The analogous results hold for the general case. 

In summary, we see analytically that the obstruction to doing local gluing comes from the cokernel of the linearized operator, and that this kernel exactly corresponds to quantities that must be (almost) conserved in the procedure.  Thus to meet this constraint, we need a good family of models at infinity for which the conserved quantities are effective parameters.  We also emphasize that we need to make good approximate solutions, which we can do here since the data tends to flat near infinity.  

Note that this general model for a gluing proceedure is well known and often used throughout geometric analysis.  The key novel feature introduced here is to exploit the underdetermined nature of the constraint equations (via the overdetermined-ellipticity of the adjoint of the linearized constraint operator) to solve for the perturbations with compact support.  

We make two further remarks.  We saw earlier that if we deform the Euclidean metric with a small compactly supported tensor $h$, we can impose the constraints in the conformal class of $g_E+h$, and the asymptotic expansion of the conformal factor has a positive coefficient of the $|x|^{-1}$-term.  In this way, we put the appropriate mass at infinity through the conformal factor.  For the asymptotic gluing, the localized deformations do not quite solve the problem: we also need to put the appropriate mass (and other charges) at infinity.  
Rather than doing this via a conformal deformation, which would induce a global perturbation, we establish this deformation by choosing an member of our model family of solutions at infinity which captures this necessary deformation in these asymptotic parameters.  Finally, we remark that this analysis has been extended to allow for the gluing of multiple asymptotic ends satisfying the vacuum constraints into a single asymptotic end, which may be interpreted as initial data for the gravitational $N$-body problem \cite{cci, cci2}.

\subsubsection{Asymptotically simple vacuum space-times} In the 1960's, Penrose \cite{pen:as} proposed a model of isolated gravitational systems (asymptotic simplicity) based on the conformal compactification of Minkowski space, cf. \cite{wa:gr}.  Indeed if we write the Minkowski metric as $\eta=-dt^2+dr^2+r^2g_{\mathbb S^2}$, where $g_{\mathbb S^2}$ is the metric on a round unit two-sphere, we can then introduce advanced and retarded null coordinates $v=t+r$, $u=t-r$, and re-scaled time and radial coordinates $T=\arctan v+ \arctan u$, $R=\arctan v -\arctan u$.  In these coordinates we obtain, for $\Omega^{-2}=\frac{1}{4}(1+v^2)(1+u^2)$,  $$\eta= \Omega^{-2} \left( -dT^2+dR^2+\sin^2(R) \; g_{\mathbb S^2}\right)= \Omega^{-2} (-dT^2+g_{\mathbb S^3}).$$ Thus we have an embedding of Minkowski space-time into the Einstein static universe $\mathbb R \times \mathbb S^3$, which is a conformal isometry.  The image of Minkowski space is a pre-compact subset, the boundary of which represents infinity, time-like, space-like, and null, of Minkowksi space-time.   It is the essential features of the construction, which represents infinity conformally faithfully but at a finite distance, that Penrose tries to capture in his definition of asymptotic simplicity.  

The Penrose proposal has had enormous influence on the study of gravitational radiation, so one would naturally like to establish the existence of a rich class of such space-times, for instance through stability results which yield new examples through perturbation.  Friedrich addressed the stability problem by rewriting the Einstein equation to emphasize the conformal structure, and he obtained a small data, semi-global stability result: for hyperboloidal data suitably close to a given hyperboloidal data set in Minkowski space (intersecting future null infinity), the resulting solution of the initial-value problem for the Einstein vacuum equation admits a conformal compactification to the future \cite{fr:as}, cf. \cite{fr:rad, fr:ga} and the more recent work of Anderson and Chru\'{s}ciel \cite{ac:as}.  On the other hand, work by Friedrich \cite{fr:ninf} and more recent work by Valiente Kroon \cite{vk:as} show that the structure of initial data near space-like infinity must indeed be special for the evolution to admit a smooth conformal compactification. 

The stability results of Friedrich provide a way to construct non-trivial asymptotically simple space-times: control the asymptotics near spatial infinity on an asymptotically flat initial data set, in such a manner that the data will evolve to a space-time with suitable hyperboloidal slices (to the future and the past), then apply the stability result to evolve from here.  In fact, Cutler and Wald \cite{cw:em} use this method to produce examples of such space-times for the Einstein-Maxwell equations.  They construct a family of time-symmetric initial data (generated by scaling a specially constructed magnetic field), which approaches Minkowski data as the mass tends to zero, such that outside a fixed ball each member of the family is identically an end of the standard asymptotically flat initial slice of a Schwarzschild space-time.  This still left open the question of whether there exist non-trivial \emph{purely radiative} space-times, i.e. non-trivial \emph{vacuum} space-times (with $\Lambda=0$) which admit a conformal compactification in the sense of Penrose.  In recent years this question has been resolved, and in fact, we have the following theorem.

\begin{Theorem}
There exists an infinite-dimensional family of solutions $(\mathbb R^3, g)$ of the time-symmetric vacuum constraint equations (with $\Lambda=0$) whose Einstein evolution is an asymptotically simple space-time, i.e. the maximal Ricci-flat space-time $(\mathbb R^4, \bar g)$ with the three-geometry $(\mathbb{R}^3, g)$ as a totally geodesic Cauchy surface admits a conformal compactification in the sense of Penrose.
\end{Theorem}

To prove this, one shows the existence of families of solutions of the time-symmetric vacuum constraints approaching the flat metric which can be perturbed to nearby solutions that are Schwarzschild outside of \textit{fixed} radius.   The difficulty which arises is that in applying the methods of \cite{cor:schw}, the resulting center of mass may drift outward as the mass tends to zero, cf. (\ref{pmap0}), and one must get estimates to control this.  Chru\'{s}ciel and Delay \cite{cd:as} first proved this statement by applying the methods of \cite{cor:schw} to produce an infinite-dimensional space of \emph{parity symmetric} solutions; under parity, the center of mass is always zero.  A different approach is taken in \cite{cor:as}, where the idea is to construct a large family of time-symmetric initial data sets  on $\mathbb R^3$ with a harmonically flat end, tending toward the Euclidean metric, for which the higher-order asymptotics are suitably bounded in terms of the mass $m$.  This allows the required estimates to scale down with $m$ tending to 0, while retaining an estimate on the center of mass.   

We briefly sketch the argument, details for which can be found in \cite{cor:as}.  We remark that the construction not only yields examples, but establishes what may be interpreted as a (weak) stability result for asymptotic simplicity at the Euclidean metric.  First recall the \emph{York} decomposition \cite{York74, Cantor77}: a symmetric tensor $h$ which decays at infinity (in a weighted space) can be written (for an appropriate vector field $X$) as $h=h^{TT}+\mathcal D(X)+\frac{1}{3}\tr(h) g_E$, where $h^{TT}$ is \emph{transverse-traceless} (trace-free and divergence free), and $\mathcal D$ is the conformal Killing operator introduced earlier, and we compute at the Euclidean metric.  Let $h$ be a compactly supported solution of the linearized scalar curvature constraint $L(h)=0$ at the Euclidean metric, so that the TT-part $h^{TT}$ from the York decomposition is non-trivial.   We note that a TT tensor with respect to the flat metric is in the kernel of $L$, and it is known that there is an infinite-dimensional space of compactly supported TT tensors at the flat metric, cf. \cite{be:tt, cor:as, df:tt}.  If we let $g_{\epsilon}=u_{\epsilon}^4 (g_E+ \epsilon h)$, and $m(\epsilon)=m(g_{\epsilon})$, then $L(h)=0$ implies $m'(0)=0$.  We can estimate the mass $m(\epsilon)$ from below, following the Brill-Deser approach to the local positivity of the ADM mass \cite{bd:pmt} (\emph{cf}. \cite{cbm:locm}).  Indeed, the second variation of the mass is given by $16 \pi m''(0)= \frac{1}{2} \int\limits_{\mathbb{R}^3} |\nabla h^{TT}|^2 \; dv_e$.  This allows us to estimate $m(\epsilon)\gtrsim \epsilon^2$, while the fact that $R(g_E+\epsilon h)\lesssim \epsilon^2$ yields a weighted estimate of $v_{\epsilon}:=(u_{\epsilon}-1)$, for large $|x|$: $|x|^2\big|v_{\epsilon}-\frac{m(\epsilon)}{2|x|}\big|\lesssim \epsilon^2 \lesssim m(\epsilon)$.  This is the key estimate to control the center of mass.  

\subsection{Conformal gluing constructions}
\label{IMPgluing} In~\cite{imp}, Isenberg, Mazzeo and Pollack
developed a gluing construction (often referred to as ``IMP gluing") for initial data sets
satisfying certain natural non-degeneracy assumptions.  The
perspective taken there is to work within the conformal
method, and thereby establish a gluing theorem for solutions of
the determined system of PDEs given by (\ref{vacmom:noncmc})
and (\ref{vacham:noncmc}).  This was initially done only within
the setting of constant mean curvature initial data sets and in
dimension $n=3$ (the method was extended to all higher
dimensions in~\cite{IMaxP}).  The construction of~\cite{imp}
allowed one to combine initial data sets by taking a connected
sum of their underlying manifolds, to add wormholes (by
performing codimension-$3$ surgery on the underlying,
connected, three-manifold) to a given initial data set, and to
replace arbitrary small neighborhoods of  points in an initial
data set with asymptotically hyperbolic ends.

In \cite{LMaz2}, building on work he had done for constant scalar curvature metrics \cite{LMaz1, LMaz3}, Mazzieri generalized the IMP gluing to the setting of the Schoen-Yau and Gromov-Lawson surgery result discussed earlier.  Namely he showed that under certain conditions, CMC solutions of the vacuum constraint equations may be combined into new solutions by gluing along a common (i.e. isometrically embedded with diffeomorphic normal bundles) submanifold of codimension greater than or equal to 3.  

In~\cite{IMP2} the IMP gluing construction was extended to only
require that the mean curvature be constant in a small
neighborhood of  the point about which one wanted to perform a
connected sum. This extension enabled the authors to show that
one can replace an arbitrary small neighborhood of a generic
point in any initial data set with an asymptotically {\em flat}
end.  Since it is easy to see that CMC solutions of the vacuum
constraint equations exist on any compact manifold~\cite{Witt},
this leads to the following result which asserts that there are
no topological obstructions to asymptotically flat solutions of
the vacuum constraint equations.

\begin{theorem}[\cite{IMP2}]
Let $\hyp$ be any closed $n$-dimensional manifold, and let $p\in\hyp$. Then $\hyp\setminus\{p\}$
admits an asymptotically flat initial data set satisfying the vacuum constraint equations.
\end{theorem}

\subsection{Initial data engineering}
\label{IDE} The gluing constructions of~\cite{imp}
and~\cite{IMP2} are performed using the determined elliptic
system provided by the conformal method, which necessarily
(due to the unique continuation property for this system)
leads to a global deformation, small away from the gluing site, of the initial data set.  Now, the ability of the Corvino-Schoen asymptotic
gluing technique to establish compactly supported deformations
invited the question of whether these conformal gluings could
be localized.  This was answered in the affirmative
in~\cite{cd} for CMC initial data under the additional,
generically satisfied~\cite{CHBeignokids}, assumption that
there are no KIDs (cokernel of the linearized operator, cf. Section \ref{sec:gc}) in a neighborhood of the gluing site.
Namely, Chru\'sciel and Delay showed that, assuming the absence of KIDs, an additional perturbation can be made to localize the IMP conformal gluing.  This left open the question of whether generic, localized gluing could be performed from the start, with making any of the assumptions (like CMC) which were present in the IMP gluing from the dependence on the conformal method.

In~\cite{CIP:PRL, cip}, this question was answered, and the IMP gluing 
was substantially improved upon,
by combining the gluing construction of~\cite{imp} together
with the Corvino-Schoen gluing technique of~\cite{cor:schw, cd},
to obtain a localized gluing construction in which the only
assumption is the absence of KIDs near points.  For a given
$n$-manifold $\hyp$ (which may or may not be connected)  and
two points $p_a\in \hyp$, $a=1,2$, we let $\tilde\hyp$ denote
the manifold obtained by replacing small geodesic balls around
these points by a neck  $\mathbb S^{n-1}\times I$. When $\hyp$ is
connected this corresponds to performing codimension-$n$
surgery on the manifold.  When the points $p_a$ lie in
different connected components of $\hyp$, this corresponds to
taking the connected sum of those components.  

\begin{Theorem}[\cite{CIP:PRL, cip}]
\label{Tlgluingv} Let $(\hyp, \threeg, K)$ be a smooth vacuum
initial data set, with $M$ not necessarily connected, and
consider two open sets $\Omega_a\subset \hyp$, $a=1,2$, with
compact closure and smooth boundary, such that
the set
of KIDs within  each $\Omega_a$  is trivial. Then for all $p_a\in \Omega_a$, $\epsilon >0$, and $k\in \N$
there exists a smooth vacuum initial data set
$(\tilde\hyp,\threeg(\epsilon),K(\epsilon))$ on the glued
manifold $\tilde\hyp$ such that
$(\threeg(\epsilon),K(\epsilon))$ is $\epsilon$-close to
$(\threeg, K)$ in a $C^k\times C^k$ topology away from
$B(p_1,\epsilon)\cup B(p_2,\epsilon)$. Moreover
$(\threeg(\epsilon),K(\epsilon))$ coincides with $(\threeg, K)$
away from $\Omega_1\cup \Omega_2$.
\end{Theorem}

This result is sharp in the following sense: first note that,
by the positive mass theorem, initial data for Minkowski
space-time cannot locally be glued to anything else  which is
non-singular and vacuum. This meshes with the fact that for
Minkowskian initial data, we have non-trivial KIDs on
any open set $\Omega$. Next,  recall that by the results
in~\cite{CHBeignokids}, the no-KID hypothesis in Theorem
\ref{Tlgluingv} is generically satisfied. Thus, the result can
be interpreted as the statement that for generic vacuum initial
data sets the local gluing can  be performed around arbitrarily
chosen points $p_a$. In particular the collection of initial
data with generic regions $\Omega_a$ satisfying the hypotheses
of Theorem~\ref{Tlgluingv} is not empty.

The proof of Theorem \ref{Tlgluingv} is a mixture of gluing
techniques developed in~\cite{IMaxP,imp} and those
of~\cite{cs:ak, cor:schw, cd}. In fact, the proof
proceeds initially via a generalization of the analysis
in~\cite{imp} to compact manifolds with boundary.  In order to
have CMC initial data near the gluing points, which the
analysis based on~\cite{imp} requires,  one makes use of the
work of Bartnik~\cite{bartnik:variational} on the plateau
problem for prescribed mean curvature spacelike hypersurfaces
in a Lorentzian manifold.

One application of Theorem \ref{Tlgluingv}
concerns the question of the existence of CMC slices in
space-times with compact Cauchy surfaces.
In~\cite{bartnik:cosmological}, Bartnik showed that there exist
maximally extended, globally hyperbolic solutions of the
Einstein equations \textit{with dust} which admit no CMC
slices. Later, Eardley and Witt (unpublished) proposed a scheme
for showing that similar vacuum solutions exist, but their
argument was incomplete. Using Theorem \ref{Tlgluingv}, we obtain: 
\begin{Corollary}[\cite{CIP:PRL, cip}]
\label{noCMCslices} There exist maximal globally hyperbolic
vacuum space-times with compact Cauchy surfaces which contain
no compact spacelike hypersurfaces with constant mean
curvature.
\end{Corollary}

Compact Cauchy surfaces with constant mean curvature are useful
objects, as the existence of one such surface gives rise to a
unique foliation  by such surfaces~\cite{BF78}, and hence a
canonical choice of time function (often referred to as \emph{CMC time} or
\emph{York time}). Foliations by CMC Cauchy surfaces have also been
extensively used in numerical analysis to explore the nature of
cosmological singularities.  With this in mind, one natural and important question  is the extent to which space-times with
no CMC slices are common among solutions to the vacuum Einstein
equations with a fixed spatial  topology. It is, in particular,
expected that the examples constructed
in~\cite{cip,CIP:PRL} are not isolated.

\subsubsection{Non-zero cosmological constant}
 \label{ssSnonzeroLambda}
Gluing constructions have also been carried out  with a
non-zero cosmological constant~\cite{ChPollack,CPP,ChDelayAH}.
These constructions yield space-times which coincide,  in the
asymptotic region, with the corresponding black hole models.  For time-symmetric
slices of these space-times, the constraint equations reduce to
the equation  for constant scalar curvature $R=2\Lambda$.
Gluing constructions have been previously carried out  in this
context, especially in the case of $\Lambda>0$, but
in~\cite{ChPollack,CPP,ChDelayAH} the emphasis is on gluing
with compact support, in the spirit of the Corvino--Schoen technique.

The  time-symmetric slices  of the $\Lambda>0$ Kottler-Schwarzschild-de Sitter
space-times provide ``Delaunay" metrics (see~\cite{ChPollack}
and references therein), and the main result  of
\cite{ChPollack,CPP} is the construction of large families of
metrics with exactly Delaunay ends. When $\Lambda<0$ the focus
is on asymptotically hyperbolic metrics with constant negative
scalar curvature. With hindsight, within the family of Kottler
metrics with $\Lambda \in \R$ (with $\Lambda<0$ corresponding to the 
Schwarzschild-anti-de Sitter metrics, $\Lambda =0$  the Schwarzschild metrics and 
$\Lambda>0$ the Schwarzschild-de Sitter metrics), the gluing in the $\Lambda>0$
setting is technically easiest, while that with $\Lambda<0$ is
the most difficult. This is due to the fact that for
$\Lambda>0$, one deals with a single linearized operator with a
one-dimensional kernel, which corresponds to changing the Delaunay parameter; in the case $\Lambda=0$, the kernel is $(n+1)$--dimensional corresponding to changes in the mass and center of mass; while for $\Lambda<0$, one needs to
consider a \emph{one-parameter family} of operators, each having $(n+1)$--dimensional kernels
\cite{ChDelayAH}.
%


\end{document}